\documentclass[%
  letterpaper,
  onecolumn,
]{mypreprint}

\usepackage[english]{babel}

\usepackage{graphicx}
\usepackage{tikz}
\usepackage{pgfplots}
  \pgfplotsset{compat = 1.13}
  \usepgfplotslibrary{external}
  \tikzset{external/system call = {%
    pdflatex \tikzexternalcheckshellescape
      -halt-on-error
      -interaction=batchmode
      -jobname "\image" "\texsource"}}
  \tikzexternaldisable

\usepackage{subcaption}

\usepackage{amsmath}
\usepackage{amssymb}
\usepackage{amsthm}

\usepackage[linesnumbered, lined, algoruled]{algorithm2e}
\Crefname{algocf}{Algorithm}{Algorithms}

\usepackage{enumitem}

\theoremstyle{plain}\newtheorem{theorem}{Theorem}
\theoremstyle{plain}\newtheorem{lemma}{Lemma}
\theoremstyle{plain}\newtheorem{proposition}{Proposition}
\theoremstyle{plain}\newtheorem{corollary}{Corollary}
\theoremstyle{definition}

\renewcommand{\rm}[1]{\ensuremath{\mathrm{#1}}}
\renewcommand{\sf}[1]{\ensuremath{\mathsf{#1}}}

\DeclareMathOperator{\rank}{rank}
\DeclareMathOperator{\mspan}{span}
\newcommand{\trans}{\ensuremath{\mkern-1.5mu\sf{T}}}

\newcommand{\R}{\ensuremath{\mathbb{R}}}
\newcommand{\N}{\ensuremath{\mathbb{N}}}

\newcommand{\nh}{\ensuremath{N}}
\newcommand{\nr}{\ensuremath{n}}
\newcommand{\nmin}{\ensuremath{n_{\min}}}
\newcommand{\nrr}{\ensuremath{r}}
\newcommand{\np}{\ensuremath{p}}
\newcommand{\nq}{\ensuremath{q}}
\newcommand{\dT}{\ensuremath{T}}
\newcommand{\nc}{\ensuremath{\nh_{\rm{c}}}}
\newcommand{\nx}{\ensuremath{\nh_{\rm{x}}}}

\newcommand{\Xcal}{\ensuremath{\mathcal{X}}}
\newcommand{\Vcal}{\ensuremath{\mathcal{V}}}
\newcommand{\Vcalt}{\widetilde{\Vcal}}

\newcommand{\Ar}{\ensuremath{\widehat{A}}}
\newcommand{\Br}{\ensuremath{\widehat{B}}}
\newcommand{\Kr}{\ensuremath{\widehat{K}}}
\newcommand{\Xr}{\ensuremath{\widehat{X}}}
\newcommand{\Sigmar}{\ensuremath{\widehat{\Sigma}}}
\newcommand{\xr}{\ensuremath{\hat{x}}}

\newcommand{\At}{\ensuremath{\widetilde{A}}}
\newcommand{\Bt}{\ensuremath{\widetilde{B}}}
\newcommand{\Kt}{\ensuremath{\widetilde{K}}}
\newcommand{\St}{\ensuremath{\widetilde{S}}}
\newcommand{\Vt}{\ensuremath{\widetilde{V}}}
\newcommand{\vt}{\ensuremath{\tilde{v}}}
\newcommand{\Xt}{\ensuremath{\widetilde{X}}}
\newcommand{\xt}{\ensuremath{\tilde{x}}}

\newcommand{\datatrip}{(U_{-}, X_{-}, X_{+})}
\newcommand{\datatripr}{(U_{-}, \Xr_{-}, \Xr_{+})}
\newcommand{\datatript}{(U_{-}, \Xt_{-}, \Xt_{+})}

\definecolor{matlabblue}{HTML}{0072BD}
\definecolor{matlaborange}{HTML}{D95319}
\definecolor{matlabyellow}{HTML}{EDB120}
\definecolor{matlabpurple}{HTML}{7E2F8E}
\definecolor{matlabgreen}{HTML}{77AC30}
\definecolor{matlablightblue}{HTML}{4DBEEE}
\definecolor{matlabred}{HTML}{A2142F}

\tikzstyle{sline} = [
  solid,
  line width = 1.5pt
]

\tikzstyle{dline} = [
  dashed,
  line width = 1.5pt
]

\newcommand{\plotfontsize}{\footnotesize}

\pgfplotscreateplotcyclelist{stateslist}{
  {matlabblue, sline},
  {matlaborange, sline},
  {matlabpurple, sline},
  {matlabgreen, sline},
  {matlabred, sline}
}

\pgfplotscreateplotcyclelist{samplelist}{
  {draw = matlabblue, fill = matlabblue, line width = 1.25pt,
    fill opacity = 0.65},
  {draw = matlaborange, fill = matlaborange, line width = 1.25pt,
    fill opacity = 0.65},
  {draw = matlabpurple, fill = matlabpurple, line width = 1.25pt,
    fill opacity = 0.65},
}

\begin{document}


\title{On the sample complexity of stabilizing linear dynamical systems
  from data}

\author[$\ast$,1]{Steffen W. R. Werner}
\affil[$\ast$]{Courant Institute of Mathematical Sciences, New York
  University, New York, NY 10012, USA}
\affil[1]{\email{steffen.werner@nyu.edu}, \orcid{0000-0003-1667-4862}}
  
\author[$\ast$,2]{Benjamin Peherstorfer}
\affil[2]{\email{pehersto@cims.nyu.edu}, \orcid{0000-0002-1558-6775}}
  
\shorttitle{Sample complexity of stabilization from data}
\shortauthor{S. W. R. Werner, B. Peherstorfer}
\shortdate{2022-07-22}
\shortinstitute{}
  

  
\abstract{%
  Learning controllers from data for stabilizing dynamical systems typically
  follows a two step process of first identifying a model and then constructing
  a controller based on the identified model.
  However, learning models means identifying generic descriptions of the
  dynamics of systems, which can require large amounts of data and extracting
  information that are unnecessary for the specific task of
  stabilization.
  The contribution of this work is to show that if a linear dynamical system has
  dimension (McMillan degree) $\nr$, then there always exist $\nr$ states from
  which a stabilizing feedback controller can be constructed, 
  independent of the dimension of the representation of the observed states and
  the number of inputs.
  By building on previous work, this finding implies that any linear dynamical
  system can be
  stabilized from fewer observed states than the minimal number of states
  required for
  learning a model of the dynamics.
  The theoretical findings are demonstrated with numerical experiments that
  show the stabilization of the flow behind a cylinder from less
  data than necessary for learning a model.
}

\novelty{}

\maketitle


\section{Introduction}%
\label{sec:intro}

Learning feedback controllers from data for stabilizing dynamical systems
typically follows a two step process:
First, a model of the underlying system of interest is identified from data.
Then, a controller is constructed based on the identified model~\cite{BruK19}.
However, learning models means identifying generic descriptions of the dynamics
of systems, which can require large amounts of data and can include extracting
information about the systems that are unnecessary for the specific task of
finding stabilizing controllers.
Additionally, if data are received in form of observed state trajectories, then
they can come in non-minimal representations in spaces of higher dimensions than
the minimal dimension of the space in which the dynamics of the systems evolve;
cf.~McMillan degree~\cite[Sec.~4.2.2]{Ant05}.
The non-minimal representation of the observed states means that higher
dimensional models are learned than necessary for describing the system, which
in turn requires even larger numbers of samples and higher training costs. 

This work focuses on the design of low-dimensional state-feedback controllers
for linear dynamical systems.
The main finding is that even if states of a system are observed in a
high-dimensional representation, the required number of states to learn a
stabilizing controller scales with the intrinsic, minimal dimension of the
system rather than the dimension of the representation of the states:
If a system has dimension $\nr$, then there exist $\nr$ states from which a
stabilizing feedback controller can be constructed (\Cref{cor:numsamp}
on page~\pageref{cor:numsamp}).
If instead only $\nr - 1$ or fewer states are observed, then there cannot exist
a feedback controller that stabilizes all systems from which the sampled states
can be observed.
This finding shows that stabilization via state-feedback can be achieved from
fewer observed states than the minimal number of states required for identifying
models, which is a consequence of~\cite{VanETetal20} and means that the
stabilization of any linear dynamical system, for which a stabilizing controller
exists, is possible with less data than learning a model.

The task of data-driven controller design roots back to~\cite{ZieN42}, which led
to model-free controller design in which controllers are learned via the
parametrization of suitable control laws that are tuned via optimization against
given data; see, for example,~\cite{CamLS02, FliJ13, LeqGMetal03, SafT95}.
The development of model reduction techniques~\cite{Ant05, BenGW15, BenSGetal21,
BenSGetal21a, QuaR14} made model-based control tractable, which allows the
application of more complex control laws than in controller tuning.
In particular, model reduction also motivated the two step approach of first
system identification and subsequent controller design because reduced models
are of lower dimensions and thus cheaper to identify; even though it is far from
guaranteed that controllers based on identified reduced models stabilize the
original system.
A large body of work has been established for learning (reduced)
dynamical-system models from data such as dynamic mode decomposition and
operator inference~\cite{Sch10, TuRLetal14, PehW16, Peh20}, sparse
identification methods~\cite{BruPK16, SchTW18, SchCHetal13}, and the Loewner
framework~\cite{BeaG12, MayA07, PehGW17, SchUBetal18, SchU16}. 
All of these methods are aiming to identify general models rather than learning
models specifically for the purpose of controller design.
In~\cite{KaiKB18, KaiKB21}, the authors take into account the task of control
when learning models and focus on nonlinear systems.
However, no sample complexity results are provided.
In~\cite{BruBPetal16}, the authors select data such that they are informative
for control with models learned via the Koopman operator, and in the
work~\cite{KraPW17} models are constructed adaptively from data for controlling
systems with quickly changing dynamics.
The authors of~\cite{DeaMMetal20, TuBPetal17} balance model approximation error
and control but aim to identify models of the same dimension as the observed
states, rather than learning low-dimensional controllers as in the present work.

The construction of low-dimensional controllers has been studied extensively by
the model reduction community; see, e.g.,~\cite{BenCQetal00, BenHW21, BenLP08,
BreMS21, JonS83}.
Such classical techniques belong to the class of approaches consisting of two
steps of first, identifying a general (reduced) model of the system from data,
followed by the controller design.
The authors of~\cite{DrmMM18, GosGB21} show that several of the classical model
reduction methods such as balanced and modal truncation are applicable even if
only data are available; however, it remains unclear how many data samples are
necessary to learn the reduced models and subsequently construct stabilizing
controllers.

The idea of data-driven controller construction regained anew interest through
the influential work~\cite{WilRMetal05}.
It introduced the so-called fundamental lemma of linear systems that states that
all trajectories of a linear system can be obtained from any given trajectory
under the assumption that the input signal is persistently exciting.
This result can be applied to study system identification, but it also led to
new approaches and strategies for controller design such as the data-driven
construction of stabilizing state-feedback
controllers~\cite{BerKSetal20, VanETetal20, DePT20}.
This line of work serves as a building block for our contribution.
We build on~\cite{VanETetal20}, which shows that fewer data samples are
sufficient for stabilization than for identifying models in certain situations;
however, the work~\cite{VanETetal20} does not consider low-dimensional
representations and operates in spaces that have the same dimension as the
observed states.
In contrast, we show that the intrinsic, minimal dimension of a system
determines how many states need to be observed for stabilization, independent of
the dimension of the data.
Key to the analysis is a combination of arguments common in model
reduction~\cite{Ant05} with a
careful distinction between the stabilizability of systems versus the
stabilizability of models of systems.
The distinction between model and system is particularly important for
data-driven control because models learned from non-minimal representations of
data are not unique and thus can be unstabilizable, which makes stabilization
via the identified models intractable independent of whether the underlying
systems are stabilizable or not.

The manuscript is organized as follows:
Preliminaries and building blocks for this work are described in
\Cref{sec:basics}.
We will carefully distinguish between models of systems and the
systems themselves.
The main contribution is \Cref{sec:infercon} that shows that the number
of observed states required for stabilization scales with the dimension of the
system rather than the dimension of the data and the model.
The case of approximately low-dimensional systems is discussed, too.
In \Cref{sec:algorithms}, we provide computational algorithms.
Numerical examples in \Cref{sec:examples} demonstrate the theory and
conclusions are drawn in \Cref{sec:conclusions}.


\section{Preliminaries}%
\label{sec:basics}

This section reviews classical results about system identification for linear
dynamical systems and discusses the concept of data informativity that was
introduced in~\cite{GevBBetal09,VanETetal20}.


\subsection{Sampling data from dynamical systems}

We consider data triplets of the form $\datatrip$.
If the system from which data are sampled is discrete in time, then state-space
models have the form
\begin{align} \label{eqn:dtsys}
  x(t + 1) & = A x(t) + B u(t),\qquad t \in \N_{0},
\end{align}
with $A \in \R^{\nh \times \nh}$ and $B \in \R^{\nh \times \np}$, and the
matrices $X_-$ and $X_+$ are
\begin{align*}
  \begin{aligned}
    X_{-} & = \begin{bmatrix} x(0) & x(1) & \ldots & x(\dT-1) \end{bmatrix}
      \in \R^{\nh \times \dT} & \text{and}\\
    X_{+} & = \begin{bmatrix} x(1) & x(2) & \ldots & x(\dT) \end{bmatrix}
      \in \R^{\nh \times \dT},
  \end{aligned}
\end{align*}
where the columns are instances of the state $x(t) \in \R^{\nh}$ and
$\N_{0} = \{0\} \cup \N$.
The inputs used to generate $X_{-}$ and $X_{+}$ are the columns of the matrix
\begin{align*}
  U_{-} & = \begin{bmatrix} u(0) & u(1) & \ldots & u(\dT - 1) \end{bmatrix}
    \in \R^{\np \times \dT}.
\end{align*}
In the case of continuous-time systems, state-space models have the form
\begin{align} \label{eqn:ctsys}
  \dot{x}(t) & = A x(t) + B u(t),\qquad t \geq 0,
\end{align}
with $A \in \R^{\nh \times \nh}$ and $B \in \R^{\nh \times \np}$ and the states
are $x(t_{0}), x(t_{1}), \ldots, x(t_{\dT - 1}) \in \R^{\nh}$ at
times $0 = t_{0} < t_{1} < \ldots < t_{\dT - 1}$.
Then, the matrix $X_-$ is
\begin{align*}
  X_{-} & = \begin{bmatrix} x(t_{0}) & x(t_{1}) & \ldots & x(t_{\dT-1}) 
    \end{bmatrix} \in \R^{\nh \times \dT},
\end{align*}
with the corresponding time derivatives 
\begin{align*}
  X_{+} & = \begin{bmatrix} \dot{x}(t_0) & \dot{x}(t_{1}) & \ldots & 
    \dot{x}(t_{\dT-1}) \end{bmatrix} \in \R^{\nh \times \dT},
\end{align*}
and inputs
\begin{align*}
  U_- & = \begin{bmatrix} u(t_0) & u(t_1) & \ldots & u(t_{\dT - 1})
    \end{bmatrix} \in \R^{\np \times \dT}.
\end{align*}

The feasible initial conditions $x(0)$ are in a
subspace $\Xcal_{0} \subset \R^{\nh}$.
Note that the space $\Xcal_{0}$ of initial conditions influences the minimal
dimension of the space in which the dynamics of the system states evolve; we
will re-visit this in detail below.

In the following, the matrices $A$ and $B$ from
models~\cref{eqn:dtsys,eqn:ctsys} are unavailable and only trajectories can be
sampled from initial conditions and inputs.


\subsection{Control via system identification}%
\label{subsec:identfb}

The matrix $K \in \R^{\np \times \nh}$ is a stabilizing controller if the system
closed with the feedback input $u(t) = K x(t)$ is asymptotically stable; see,
e.g.,~\cite{Dat04, DraH99}.
Consequently, the system is called \emph{stabilizable} if such a state-feedback
matrix $K$ exists.

Stabilizability can also be described in terms of models as
follows: A discrete-time model~\cref{eqn:dtsys} is called
\emph{stabilizable} if there exists a feedback matrix $K$ such that the
eigenvalues of $A + B K$ are in the open unit disk.
A continuous-time model~\cref{eqn:ctsys} is called
\emph{stabilizable} if there exists a feedback matrix $K$ such that the
eigenvalues of $A + B K$ are in the open left half-plane.
A system is stabilizable if and only if there exists a model of the system that
is stabilizable.

One approach for deriving a controller $K$ from data is first identifying a
model from a data triplet $\datatrip$ and then applying classical
control approaches to construct a $K$ from the identified model.
However, identifying a model can be expensive in terms of
number of data samples $\dT$ that are required.
The following proposition states the necessary condition for identifying
state-space models and a constructive approach to do so.

\begin{proposition}[Identification of state-space model~\cite{VanD96}]%
  \label{prp:sysident}
  Let $\datatrip$ be a data triplet. 
  The underlying state-space model~\cref{eqn:dtsys} (or~\cref{eqn:ctsys}) can
  be uniquely identified from the data triplet as
  \begin{align*}
    \begin{aligned}
      A & = X_{+} V_{1}^{\dagger} & \text{and} &&
        B & = X_{+} V_{2}^{\dagger},
    \end{aligned}
  \end{align*}
  if and only if
  \begin{align}
    \rank\left( \begin{bmatrix} X_{-} \\ U_{-} \end{bmatrix} \right) &
      = \nh + \np,\label{eq:PropSysIdRankCondition}
  \end{align}
  where $\begin{bmatrix} V_{1}^{\dagger} & V_{2}^{\dagger} \end{bmatrix}$ is a
  right inverse in the sense of
  \begin{align*}
    \begin{bmatrix} X_{-} \\ U_{-} \end{bmatrix}
      \begin{bmatrix} V_{1}^{\dagger} & V_{2}^{\dagger} \end{bmatrix} & =
      \begin{bmatrix} I_{\nh} & 0 \\ 0 & I_{\np} \end{bmatrix}.
  \end{align*}
\end{proposition}

Note that the identified state-space model in
\Cref{prp:sysident} is independent of the right inverse
$\begin{bmatrix} V_{1}^{\dagger} & V_{2}^{\dagger} \end{bmatrix}$.
Once a model is found, classical methods for system stabilization
such as pole assignment~\cite{Dat04}, Bass' algorithm~\cite{Arm75, ArmR76},
Riccati equations~\cite{LanR95} and partial
stabilization~\cite{BenCQetal00} are applicable.
A consequence of \Cref{prp:sysident} is that at least
$\dT = \nh + \np$ data samples are needed to identify the model
from a data triplet $\datatrip$, otherwise the rank
condition~\cref{eq:PropSysIdRankCondition} cannot be satisfied.
In particular, the dimension $\nh$ of the states of the sampled trajectory
enters in the number of required data samples and the state dimension can be
high.
Also note that the necessary condition in \Cref{prp:sysident}
can only be satisfied if sufficiently many linearly independent states are
observed.
A sufficient condition to guarantee the existence of appropriate data
samples is controllability of the unknown model.


\subsection{Inferring controllers without system identification}%
\label{subsec:inferfb}

The data informativity concept was orignally developed for system
identification~\cite{GevBBetal09}.
It was extended in~\cite{VanETetal20} to data-driven controller design
and shows that fewer than $\nh + \np$ data samples can be sufficient for
learning a stabilizing controller $K$.
Consider the set of state-space models that explain a given data triplet
$\datatrip$
\begin{align*}
  \Sigma_{\rm{i/s}} & := \left\{ (A, B) \left\lvert~ X_{+} = A X_{-} + B
    U_{-} \right. \right\}.
\end{align*}
There can be many state-space models of a single system that explain the data
$\datatrip$ in the sense of 
\begin{equation}
X_+ = AX_- + BU_-.
\label{eq:prelim:explainData}
\end{equation}
Additionally, there can be different systems that explain a data triplet.

Let further
\begin{align*}
  \Sigma_{K} & := \left\{ (A, B) \left\lvert~ A + B K
    \text{ is asymptotically stable} \right. \right\}
\end{align*}
be the set of state-space models that are stabilized by a given
controller $K$.
If there exists a $K$ such that $\Sigma_{\rm{i/s}} \subseteq
\Sigma_{K}$ holds, then the data triplet $\datatrip$ is called informative for
stabilization by state feedback; see~\cite{VanETetal20} for details.
In other words, the data triplet is informative for stabilization by
feedback if and only if there exists a stabilizing controller that
stabilizes all state-space models and thus all systems that explain the data in
the sense of~\cref{eq:prelim:explainData}.

\begin{proposition}[Data informativity in discrete time~\cite{VanETetal20}]%
  \label{prp:constfeed}
  Let $\datatrip$ be a data triplet sampled from a discrete-time state-space
  model.
  The data triplet is informative for stabilization if and only if one of the
  following two equivalent statements holds:
  \begin{enumerate}
    \item The matrix $X_{-}$ has full row rank and there exists a right
      inverse $X_{-}^{\dagger}$ of $X_{-}$ such that $X_{+}
      X_{-}^{\dagger}$ is (discrete-time) asymptotically stable.
      The controller is then given by $K = U_{-} X_{-}^{\dagger}$
      that satisfies $\Sigma_{\rm{i/s}} \subseteq \Sigma_{K}$.
    \item There exists a matrix $\Theta \in \R^{\dT \times \nh}$ such that
      \begin{align} \label{eqn:dtlmi}
        \begin{aligned}
          X_{-} \Theta & = (X_{-} \Theta)^{\trans} & \text{and} &&
            \begin{bmatrix} X_{-} \Theta & X_{+} \Theta \\
            (X_{+} \Theta)^{\trans} & X_{-} \Theta \end{bmatrix} > 0.
        \end{aligned}
      \end{align}
      The controller is then given by $K = U_{-} \Theta (X_{-} \Theta)^{-1}$ 
      that satisfies $\Sigma_{\rm{i/s}} \subseteq \Sigma_{K}$.
  \end{enumerate}
\end{proposition}

\begin{corollary}[Data informativity in continuous time]%
  \label{cor:constfeed}
  Let $\datatrip$ be a data triplet sampled from a continuous-time state-space
  model.
  The data triplet is informative for stabilization if and only if one of the
  following two equivalent statements holds:
  \begin{enumerate}
    \item The matrix $X_{-}$ has full row rank and there exists a right
      inverse $X_{-}^{\dagger}$ of $X_{-}$ such that $X_{+}
      X_{-}^{\dagger}$ is (continuous-time) asymptotically stable.
      The controller is then given by $K = U_{-} X_{-}^{\dagger}$
      that satisfies $\Sigma_{\rm{i/s}} \subseteq \Sigma_{K}$.
    \item There exists a matrix $\Theta \in \R^{\dT \times \nh}$ such that
      \begin{align} \label{eqn:ctlmi}
        \begin{aligned}
          X_{-} \Theta & > 0 & \text{and} &&
            X_{+} \Theta + \Theta^{\trans} X_{+}^{\trans} < 0.
        \end{aligned}
      \end{align}
      The controller is then given by $K = U_{-} \Theta (X_{-} \Theta)^{-1}$ 
      that satisfies $\Sigma_{\rm{i/s}} \subseteq \Sigma_{K}$.
  \end{enumerate}
\end{corollary}
\begin{proof}
  The proof follows directly from the discrete-time case in
  \Cref{prp:constfeed} and the con\-tin\-u\-ous-time conditions for
  data-based feedback construction in~\cite[Remark~2]{DePT20}.
\end{proof}

The condition on the full row rank of $X_{-}$ in
\Cref{prp:constfeed,cor:constfeed} implies that at least $\nh$ data samples
are needed for feedback construction from observed states in general, which is
$\np$ fewer states than minimally required for identifying a model;
cf. \Cref{prp:sysident}.
However, the minimal number of data samples $\nh$ still depends on the dimension
of the sampled states, which is potentially high; in particular, for dynamical
systems stemming from discretizations of partial differential equations.


\section{Inferring low-dimensional controllers from high-dimensional states}%
\label{sec:infercon}

In this section, we establish the sample complexity for constructing stabilizing
controllers with high-dimensional state samples from intrinsically
low-dimensional systems.
We show that if the system of interest has intrinsic dimension $\nr$, then there
exist $\nr$ states from which a stabilizing feedback controller can be
constructed.
This is in contrast to the results surveyed in \Cref{sec:basics}, where
the number of data samples scales with the dimension of the observed states
rather than the intrinsic dimension of the system of interest.
We show further that a strictly lower number of samples than the intrinsic
dimension of a system is insufficient for finding controllers that stabilizes
all systems from which the observed states can be sampled.
Thus, if only $\nr - 1$ or fewer states are observed, then there cannot be a
feedback controller $K$ that stabilizes all systems that can produce the
observed states; in particular, a constructed $K$ might not stabilize the actual
system of interest from which data have been sampled.


\subsection{Controller inference for stabilizing intrinsically low-dimensional
  systems}%
\label{subsec:lowdimmod}

In this section, we consider low dimensional systems.
Recall that $\nh$ is the dimension of the states that are sampled
from a model of the system of interest. The sampled states define the data
triplet $\datatrip$.
A system is called low dimensional if there exists an $\nr \in \N$ with
$\nr < \nh$ and a full-rank matrix $V \in \R^{\nh \times \nr}$ such that for
all initial conditions $x_{0} \in \Xcal_{0}$ and any inputs
$u(t) \in \R^{\np}$ there exist reduced states
$\xr(t) \in \R^{\nr}$ of dimension $\nr$ with
\begin{align} \label{eqn:CInf:IntDim}
  \begin{aligned}
    x(t) & = V \xr(t), & \text{for all}~t \geq 0.
  \end{aligned}
\end{align}
Equivalently, since $V$ has full rank, this means that there are
$\nr$-dimensional state-space models of the system with states $\xr(t)
\in \R^{\nr}$ satisfying~\cref{eqn:CInf:IntDim}.
The intrinsic (minimal) state-space dimension $\nmin$ of the system, i.e., the
smallest state-space dimension of $\xr(t)$ such that~\cref{eqn:CInf:IntDim}
holds, is uniquely determined.
The minimal dimension $\nmin$ depends on the
controllability of the corresponding state-space realizations and on the initial
conditions from $\Xcal_{0}$. The minimal dimension $\nmin$ coincides with the
McMillan degree of the system if the space of initial conditions $\Xcal_0$ has
dimension 0, $\Xcal_0 = \{0\}$; cf.~\cite[Sec.~4.2.2]{Ant05}.
In many applications, the states describe the
deviation from a desired steady state and then considering only initial
condition $\Xcal_0 = \{0\}$ is a common choice.
However, for example, if $\Xcal_{0} = \R^{\nh}$, then any possible state in
$\R^{\nh}$ can be reached as initial condition and then the state-space model is
minimal independent of its controllability. 

In the following, we refer to $V$ as basis matrix, to $\nh$ as the
high dimension and to $\nr \geq \nmin$ as the reduced dimension.
Note that $\nr$ does not have to be the minimal dimension $\nmin$.

  
\subsubsection{Lifting controllers}

The Kalman controllability form of a state-space model will be helpful in
the following.
For our purposes, we will use the following variant:
For a state-space model~\cref{eqn:dtsys} (or~\cref{eqn:ctsys}) and
basis matrix $X_{0} \in \R^{\nh \times \nq}$ of the initial conditions'
subspace, i.e., $X_{0}$ is full-rank and the span of the columns of $X_0$ is
$\Xcal_{0}$, there exists an invertible $S \in \R^{\nh \times \nh}$ such that
\begin{align} \label{eqn:extcontform}
  \begin{aligned}
    S^{-1} A S & = \begin{bmatrix} A_{11} & A_{12} & A_{13}\\ 0 & A_{22} &
      A_{23}\\ 0 & 0 & A_{33} \end{bmatrix}, &
    S^{-1} B & = \begin{bmatrix} B_{1} \\ 0 \\ 0 \end{bmatrix},& 
    S^{-1} X_{0} & = \begin{bmatrix} X_{10} \\ X_{20} \\ 0 \end{bmatrix}
  \end{aligned}
\end{align}
with the matrix blocks $A_{11} \in \R^{\nc \times \nc}$,
$A_{12} \in \R^{\nc \times \nx}$,
$A_{13} \in \R^{\nc \times (\nh - \nc - \nx)}$,
$A_{22} \in \R^{\nx \times \nx}$,
$A_{23} \in \R^{\nx \times (\nh - \nc - \nx)}$,
$A_{33} \in \R^{(\nh - \nc - \nx) \times (\nh - \nc - \nx)}$,
$B_{1} \in \R^{\nc \times \np}$,
$X_{10} \in \R^{\nc \times \nq}$,
$X_{20} \in \R^{\nx \times \nq}$,
where the dimension $\nh - \nc - \nx$ of the last block row is maximal; see,
e.g.,~\cite{Ros70, Voi15}.
The first block row in~\cref{eqn:extcontform} is the controllable part
of the system, which can be influenced by the control inputs $u(t)$.
The corresponding dimension of the controllability subspace is given by the
size $\nc \in \N_{0}$ of the block matrices.
Similarly, the second block row in~\cref{eqn:extcontform} corresponds to the
system components that cannot be controlled but are steered by the
initial conditions.
The corresponding dimension is denoted by $\nx \in \N_{0}$.
The last block row of~\cref{eqn:extcontform} describes the components of the
state that are neither excited by inputs nor by the initial condition.
In the state-space model~\cref{eqn:extcontform}, they remain zero over
time, independent of $A_{33}$.

The following lemma relates the dimensions of the blocks in the
form~\cref{eqn:extcontform} to low-dimensional state spaces.

\begin{lemma}[Low-dimensional subspaces and state-space dimensions]%
  \label{lmm:minsys}
  Let $V_{\min} \in \R^{\nh \times \nmin}$ be a basis matrix such
  that~\cref{eqn:CInf:IntDim} holds with $\nmin$ the minimal dimension of the
  underlying system and $\Vcal_{\min}$ the corresponding subspace.
  For all basis matrices $V \in \R^{\nh \times \nr}$ that
  satisfy~\cref{eqn:CInf:IntDim}, with corresponding subspaces $\Vcal$, it
  holds that
  \begin{align*}
    \Vcal_{\min} \subseteq \Vcal,
  \end{align*}
  and that
  \begin{align*}
    \nc + \nx = \nmin \leq \nr \leq \nh,
  \end{align*}
  where $\nc$ and $\nx$ are the block matrix sizes from~\cref{eqn:extcontform}.
  In the special case of homogeneous initial conditions,
  $\Xcal_{0} = \{0\}$, the lower bound on the dimensions simplifies to
  \begin{align*}
    \nc = \nmin \leq \nr \leq \nh.
  \end{align*}
\end{lemma}
\begin{proof}
  For the proof, we first have a look 
  at~\cref{eqn:extcontform} since any state-space
  model can be transformed into that form.
  All states of~\cref{eqn:extcontform} can be written as
  \begin{align} \label{eqn:extcontstate}
    \xt(t) & = \begin{bmatrix} x_{\rm{c}}(t) \\
      x_{\rm{x}}(t) \\ 0 \end{bmatrix},
  \end{align}
  partitioned according to the block structure of~\cref{eqn:extcontform}.
  Due to the inputs spanning a $p$-dimensional subspace and the initial
  conditions taken from $\Xcal_{0}$, the set of all states of the system
  associated with~\cref{eqn:extcontform} is a subspace.
  In particular, the set of partitioned states $x_{\rm{c}}(t)$ is an
  $\nc$-dimensional and of $x_{\rm{x}}(t)$ an $\nx$-dimensional subspace, since
  otherwise the dimension of the last block row in~\cref{eqn:extcontform} is
  not maximal.
  With concatenation of the partitioned states in~\cref{eqn:extcontstate}, the
  minimal state-space dimension of the system associated
  with~\cref{eqn:extcontform} is given by
  \begin{align*}
    \nmin & = \nc + \nx.
  \end{align*}
  Also, from~\cref{eqn:extcontstate} it follows that there exists a basis
  matrix $\Vt_{\min} \in \R^{\nh \times \nmin}$ such that
  $\xt(t) = \Vt_{\min} \hat{\xt}(t)$, where $\hat{\xt}(t) \in \R^{\nmin}$
  is the state of a minimal state-space model of the system.
  For any other basis $\Vt \in \R^{\nh \times \nr}$, with $\xt(t) = \Vt
  \hat{\xt}_{2}(t)$, it must hold that
  \begin{align*}
    \mspan(\Vt_{\min}) & \subseteq \mspan(\Vt),
  \end{align*}
  since otherwise there are states $\xt(t)$ in $\mspan(\Vt_{\min})$ that do not
  yield the equality $\xt(t) = \Vt \hat{\xt}_{2}(t)$.
  Consequently, the results of the lemma hold for~\cref{eqn:extcontform}.
  By restoring the original states of the order $\nh$ state-space model using
  $x(t) = S \xt(t)$ and observing that the transformation $S$ does not change
  the dimensions of subspaces nor inclusion arguments, the results hold.
\end{proof}

\begin{theorem}[Lifting controllers]%
  \label{thm:lowdimfeed}
  Consider a stabilizable system from which states with dimension $\nh$ can be
  sampled. 
  Let now $V \in \R^{\nh \times \nr}$ be a basis matrix with $\nr \leq \nh$
  for which~\cref{eqn:CInf:IntDim} holds.
  Let further $\Kr \in \R^{\np \times \nr}$ be a stabilizing controller of the
  system if applied as feedback to the low-dimensional states $\xr(t)$.
  Then, for any left inverse $V^{\dagger}$ of $V$, the matrix
  $K = \Kr V^{\dagger}$ stabilizes the system if it is applied as feedback
  controller to the high-dimensional states $x(t)$.
\end{theorem}

Before we continue to the proof of \Cref{thm:lowdimfeed}, we
discuss its results first.
The theorem states that for any left inverse $V^{\dagger}$, the lifted
controller $K = \Kr V^{\dagger}$ stabilizes the system in the sense that there
exists an $\nh$-dimensional state-space model $(A, B)$ of the system such that
the matrix $A + B K$ is stable.
Similarly, the low-dimensional controller $\Kr$ stabilizes the system in the
sense that there exists a state-space model $(\Ar, \Br)$ obtained by the
basis matrix $V$ and the chosen left inverse $V^{\dagger}$ from a
high-dimensional state-space model $(\bar{A}, \bar{B})$, which is potentially
different from $(A, B)$, such that the closed-loop matrix
$\Ar + \Br \Kr$ is asymptotically stable.
In the case of $\nr > \nmin$, there might be unstabilizable $\nr$-dimensional
state-space models for which no stabilizing controller $\Kr$ can be constructed.
However, since the underlying system is stabilizable, there have to exist
stabilizable $\nr$-dimensional state-space models that describe the same
system.
If instead $\nr = \nmin$, then $(\Ar, \Br)$ is uniquely determined by
$(\bar{A}, \bar{B})$ and $V$. And, if $(\bar{A}, \bar{B})$ is a model of a
stabilizable system then $(\Ar, \Br)$ is guaranteed to be a stabilizable model.
Therefore, $\Kr$ depends only on the choice of $V$ if $\nr = \nmin$ but
additionally on $V^{\dagger}$ if $\nr > \nmin$.

\begin{proof}[Proof of \Cref{thm:lowdimfeed}.]
  We consider the extended controllability form~\cref{eqn:extcontform}
  of the unknown underlying $\nh$-dimensional state-space model.
  Without loss of generality we assume that
  \begin{align} \label{eqn:AB}
    (A, B)
  \end{align}
  is a stabilizable state-space model, since there must exist a stabilizable
  model for the underlying stabilizable system.
  Consequently, only $A_{11}$ has unstable eigenvalues.
  Let 
  \begin{equation}
  K = \begin{bmatrix} \Kt_{1} & \Kt_{2} & \Kt_{3} \end{bmatrix} T^{-1}
  \label{eq:TH31Auxiliary3}
  \end{equation}
  be
  a feedback matrix.
  The closed-loop matrix of~\cref{eqn:extcontform} is then given by
  \begin{align}
    & \begin{bmatrix} A_{11} & A_{12} & A_{13}\\ 0 & A_{22} &
      A_{23}\\ 0 & 0 & A_{33} \end{bmatrix} +
      \begin{bmatrix} B_{1} \\ 0 \\ 0 \end{bmatrix}
      \begin{bmatrix} \Kt_{1} & \Kt_{2} & \Kt_{3} \end{bmatrix}\notag\\
    & = \begin{bmatrix} A_{11} + B_{1} \Kt_{1} & A_{12} + B_{1} \Kt_{2} &
      A_{13} + B_{1} \Kt_{3}\\ 0 & A_{22} &
      A_{23}\\ 0 & 0 & A_{33} \end{bmatrix}.
      \label{eq:TH31:HelperAux1}
  \end{align}
  The eigenvalues of only the controllable block row in $A_{11}$
  are influenced by the feedback.
  Thus, the feedback $K$ stabilizes the underlying system if and only if
  $A_{11} + B_{1} \Kt_{1}$ is asymptotically stable, because the eigenvalues of
  a block triangular matrix are the union of the eigenvalues of the diagonal
  blocks.
  We now consider a case distinction on the considered reduced dimension $\nr$.
  
  Case 1 with $\nr = \nmin$: Let $V_{\min} \in \R^{\nh \times \nmin}$ be a basis matrix of the
  smallest subspace such that $x(t) = V_{\min} x_{\min}(t)$ holds for all
  $t \geq 0$, with $x(t)$ the state of~\cref{eqn:AB}.
  Since the underlying system is stabilizable, there exists a
  stabilizing feedback $K_{\min}$ for the minimal state-space model
  $(A_{\min}, B_{\min})$ associated with $V_{\min}$ and $V_{\min}^{\dagger}$ by
  \begin{align} \label{eqn:minrel}
    \begin{aligned}
      A_{\min} & = V_{\min}^{\dagger} A V_{\min}, &
        B_{\min} & = V_{\min}^{\dagger} B.
    \end{aligned}
  \end{align}
  Then, we know from \Cref{lmm:minsys} that $\nmin = \nc + \nx$.
  Also, by truncating the zeros in~\cref{eqn:extcontstate}, there must exist
  a transformation $\St$ such that
  \begin{align} \label{eqn:minblock}
    \begin{aligned}
      \St^{-1} V_{\min}^{\dagger} A V_{\min} \St & =
        \begin{bmatrix} A_{11} & A_{12} \\ 0 & A_{22} \end{bmatrix}, &
        \St^{-1} V_{\min}^{\dagger} B & = 
        \begin{bmatrix} B_{1} \\ 0  \end{bmatrix},
    \end{aligned}
  \end{align}
  and
  \begin{align}
    \begin{aligned}
      K V_{\min} \St = K_{\min} \St & =
        \begin{bmatrix} \Kt_{1} & \Kt_{2} \end{bmatrix}
        \label{eq:THM31Auxiliary2}
    \end{aligned}
  \end{align}
  holds, with $A_{11}$, $A_{12}$, $A_{22}$ and $B_{1}$
  from~\cref{eqn:extcontform}. 
  Note that~\cref{eq:THM31Auxiliary2} connects the blocks $\Kt_{1}$ and
  $\Kt_{2}$ of the transformed $K$ defined in~\cref{eq:TH31Auxiliary3} to
  $K_{\min}$.
  Since the blocks $A_{11}$ and $A_{22}$ in~\cref{eqn:minblock} are the same as
  in~\cref{eqn:extcontform}, their eigenvalues are part of the spectrum of the
  $A$ matrix in~\cref{eqn:AB}. As consequence, $V_{\min}$ spans the same space
  as the eigenvectors of $A$ corresponding to the eigenvalues of the blocks
  $A_{11}$ and $A_{22}$; cf.~deflation in the Arnoldi process described
  in~\cite[Eq.~(10.5.2)]{GolV13}.
  Consider for the sake of the argument~\cref{eqn:AB} to be discrete in time,
  then it holds
  \begin{align*}
    x(t + 1) = A x(t) + B u(t) = A V_{\min} x_{\min}(t) + B u(t)
      = V_{\min} A_{\min} x_{\min}(t) + B u(t),
  \end{align*}
  where the last equality holds because $V_{\min}$ spans the same space as
  eigenvectors of $A$,   which leads to
  \begin{align} \label{eq:ProofThm31Aux4}
    \begin{aligned}
      x_{\min}(t + 1) = V_{\min}^{\dagger} x(t + 1)
        & = A_{\min} x_{\min}(t) + V_{\min}^{\dagger} B u(t)\\
      & = A_{\min} x_{\min}(t) + B_{\min} u(t),
    \end{aligned}
  \end{align}
  for all left inverses $V_{\min}^{\dagger}$ of $V_{\min}$. Now consider a
  different left inverse $\bar{V}_{\min}^{\dagger}$ of $V$, which leads to
  \begin{align} \label{eq:ProofThm31Aux5}
    \begin{aligned}
      x_{\min}(t + 1)  = \bar{V}_{\min}^{\dagger} x(t + 1)
        & = A_{\min} x_{\min}(t) + \bar{V}_{\min}^{\dagger} B u(t)\\
      & = A_{\min} x_{\min}(t) + \bar{B}_{\min} u(t),
    \end{aligned}
  \end{align}
  and thus subtracting~\cref{eq:ProofThm31Aux4} from~\cref{eq:ProofThm31Aux5}
  leads to $B_{\min}u(t) - \bar{B}_{\min}u(t) = 0$.
  Therefore, $(A_{\min}, B_{\min})$ is, in fact, independent of the
  left inverse.
  The same line of arguments holds in the continuous-time case.  Since
  $K_{\min}$ is stabilizing for $(A_{\min}, B_{\min})$, $\Kt_{1}$
  is such that $A_{11} + B_{1} \Kt_{1}$ in~\cref{eq:TH31:HelperAux1} is
  asymptotically stable.
  It follows from above that $K$ must be a stabilizing feedback for the
  underlying system when applied to the state of~\cref{eqn:AB}.
  
  Case 2 with $\nr > \nmin$:
  From \Cref{lmm:minsys}, we know that for all $V \in \R^{\nh \times \nr}$
  that satisfy~\cref{eqn:CInf:IntDim}, the corresponding subspaces satisfy
  $\mspan(V_{\min}) \subseteq \mspan(V)$.
  Therefore, there must be a transformation $\St_{2} \in \R^{\nr \times \nr}$
  such that $V = \begin{bmatrix} V_{\min} & \Vt \end{bmatrix} \St_{2}^{-1}$,
  with an auxiliary basis matrix $\Vt$,
  and for all left inverses it holds that
  \begin{align*}
    V^{\dagger} & = \St_{2} \begin{bmatrix} V_{\min}^{\dagger} \\
      \Vt^{\dagger} \end{bmatrix}.
  \end{align*}
  It follows that any $\nr$-dimensional state-space model $(\Ar, \Br)$
  associated with the choice of $V$ and $V^{\dagger}$ can be transformed such
  that
  \begin{align*}
    \begin{aligned}
      \Ar & = \begin{bmatrix} A_{\min} & \At_{12} \\ 0 & \At_{22} \end{bmatrix},
        & \Br & = \begin{bmatrix} B_{\min} \\ 0  \end{bmatrix},
    \end{aligned}
  \end{align*}
  where $A_{\min}$ and $B_{\min}$ are the matrices from the minimal state-space
  model~\cref{eqn:minrel}, and $\At_{12}$ and $\At_{22}$ are auxiliary matrices
  depending on $\Vt$.
  Due to~\cref{eqn:AB} being a stabilizable model, we can choose $\Vt$
  such that $\At_{22}$ is asymptotically stable.
  Thus, there exists a feedback $\Kr$ that stabilizes the state-space model
  and, consequently, the system if applied to the low-dimensional states
  $\xr(t)$.
  In particular, via the same transformation that has been used for
  $(\Ar, \Br)$ it holds that $$\Kr = \begin{bmatrix} K_{\min} & \Kt_{2}
  \end{bmatrix},$$ where $K_{\min}$ must be stabilizing for the minimal
  state-space model~\cref{eqn:minrel} determined only by $V_{\min}$ via
  truncation from the $\nh$-dimensional model $(A,B)$ defined in~\cref{eqn:AB}.
  From~\cref{eq:THM31Auxiliary2},
  it follows that $K = \Kr V^{\dagger}$ stabilizes the
  system if applied to $x(t)$ independent of the choice of $V^{\dagger}$.
\end{proof}

\Cref{thm:lowdimfeed} shows the stabilization of the system via $K$ to be
independent of the chosen $V^{\dagger}$.
In fact, it can be shown that the spectral effects of $K$ only depend on $\Kr$.
This is stated by the following corollary.

\begin{corollary}[Spectrum of closed-loop matrices]%
  \label{cor:spectrum}
  Given the same assumptions as in \Cref{thm:lowdimfeed}, let $(A, B)$
  be a state-space model for $x(t)$ and $K = \Kr V^{\dagger}$ a
  stabilizing controller.
  Then, the spectrum of $A + B K$ is the same for all left inverses
  $V^{\dagger}$.
\end{corollary}
\begin{proof}[Proof.]
  The result follows directly from the use of~\cref{eqn:extcontform} in
  the proof of \Cref{thm:lowdimfeed}.
  Only the spectrum corresponding to the controllable system part can be
  influenced by the feedback $K$ or $\Kr$, respectively.
  The freedom of choosing $V^{\dagger}$ only influences the realization of the
  order-$\nh$ feedback matrix $K$, which does not result in any changes to the
  spectrum of the closed-loop matrix.
\end{proof}

While the effect of $K$ on the spectrum of the underlying closed-loop
matrix is uniquely determined by $\Kr$ independent of $V^{\dagger}$, the
realizations of $K$ as well as $\Kr$ depend on the choice of $V$ and
$V^{\dagger}$.
An advantageous choice for $V^{\dagger}$ is the Moore-Penrose inverse
$V^{+}$ of $V$ due to its simplicity of computation.
Numerically, it is often advantageous to choose $V$ with orthonormal columns
due to the numerical properties of its optimal condition number.
In this case, it holds that $V^{+} = V^{\trans}$.


\subsubsection{Inferring low-dimensional controllers}

We now show that if a space with basis $V$ and dimension $\nr$ exists such
that~\cref{eqn:CInf:IntDim} holds, then there exist $\dT = \nr$ states
that are sufficient to find a stabilizing controller even if the states are
observed in representations of higher dimension $\nh > \nr$.

First, consider the rank conditions for system identification and data
informativity in \Cref{prp:sysident,prp:constfeed,cor:constfeed}, respectively.
From the previous section we know that the full-rank conditions cannot be
satisfied for data triplets $\datatrip$ sampled from state-space models with
$\nh > \nr$.
This can be seen directly in~\cref{eqn:extcontform} because states
corresponding to the $A_{33}$ block are constant over time and thus lead to a
lower rank than $\nh$.
Information about the state-space model from observed states can only be
obtained for the first two block rows and columns in~\cref{eqn:extcontform},
which are associated with controllability and the effect of the initial
conditions.
Therefore, there exist many non-stabilizable state-space models
of dimension $\nh > \nr$ that explain the data in the sense
of~\cref{eq:prelim:explainData} and that describe the same underlying
system.
However, in this work, we are interested in the construction of controllers that
stabilize the underlying system.
This is independent of the used state-space models, i.e., we can restrict the
set of state-space models that explain the data $\Sigma_{\rm{i/s}}$ to
those with dynamics that evolve in a common low-dimensional subspace spanned by the
columns of $V \in \R^{\nh \times \nr}$.
To this end, we introduce the following set of state-space models
that explain the data triplet $\datatrip$, have low-dimensional representations
with~\cref{eqn:CInf:IntDim} for a fixed basis matrix $V$ and are stable in the
components that do not contribute to the system dynamics:
\begin{align} \label{eqn:lowdimmodels}
  \begin{aligned}
    \Sigma_{\rm{i/s}}^{\rm{s}}(V) & := \Sigma_{\rm{i/s}} \cap
      \left\{ (A, B) \left\lvert~ \exists (\Ar, \Br) ~\text{which
      satisfy~\cref{eqn:CInf:IntDim} with}~V \right. \right\}\\
    & \phantom{{}:={}\Sigma_{\rm{i/s}}} \cap
      \left\{ (A, B) \left\lvert~ V_{\perp}^{\dagger} A V_{\perp} ~\text{is stable}
      \right. \right\}\,.
  \end{aligned}
\end{align}
The columns of the basis matrix $V_{\perp}$ span the orthogonal complement of the
space spanned by the columns of $V$ such that
$\mspan\left(\begin{bmatrix} V & V_{\perp}\end{bmatrix} \right) =
\R^{\nh}$, $V^{\trans} V_{\perp} = 0$ and $V_{\perp}^{\trans} V = 0$.
The first intersection in~\cref{eqn:lowdimmodels} ensures that
$\Sigma_{\rm{i/s}}^{\rm{s}}(V)$ contains only those models of
$\Sigma_{\rm{i/s}}$ that have dynamics evolving in the same subspace spanned by the
columns of $V$.
In the case of low-dimensional systems, i.e., $\nr < \nh$, this means there are
components of the models that describe zero-dimensional dynamics and do not
contribute to the system dynamics; cf. the third
block row in~\cref{eqn:extcontform}.
These components can be described by models with system matrices with arbitrary
spectrum but the eigenvalues of the corresponding block $A_{33}$
in~\cref{eqn:extcontform} do not play a role for the stabilization of the underlying
system dynamics.
This motivates the second intersection, which filters out models with unstable
components that do not contribute to the system dynamics.
Also, we will work in the following with reduced data triplets $\datatripr$ that
have potentially a smaller state-space dimension than $\datatrip$.
For notational convenience, we extend the existing notation of the sets of 
state-space models used so far by the following:
\begin{align*}
  \Sigmar_{\rm{i/s}} & := \left\{ (\Ar, \Br) \left\lvert~ \Xr_{+} =
    \Ar \Xr_{-} + \Br U_{-} \right.\right\},\\
  \Sigmar_{\Kr} & := \left\{ (\Ar, \Br) \left\lvert~ \Ar + \Br \Kr
    \text{ is asymptotically stable} \right. \right\}.
\end{align*}

\begin{theorem}[Data informativity for low-dimensional feedback]%
  \label{thm:lowdatainform}
  Let $\datatrip$ be a data triplet sampled from a state-space model of
  dimension $\nh$ for which~\cref{eqn:CInf:IntDim} holds with $V \in
  \R^{\nh \times \nr}$.
  There exists a controller $K$ such that $\Sigma_{\rm{i/s}}^{\rm{s}}(V) \subseteq
  \Sigma_{K}$ if and only if  the data triplet $\datatripr$ is informative for
  stabilization by feedback, i.e., $\Sigmar_{\rm{i/s}} \subseteq \Sigmar_{\Kr}$,
  where $X_{-} = V \Xr_{-}$ and $X_{+} = V \Xr_{+}$.
  A stabilizing high-dimensional controller is then given by
  $K = \Kr V^{\dagger}$ for all left inverses $V^{\dagger}$ of $V$.
\end{theorem}

If \Cref{thm:lowdatainform} applies, the construction of a
$\Kr$ follows from using \Cref{prp:sysident,prp:constfeed}
or \Cref{cor:constfeed} for the reduced data triplet $\datatripr$.
The order $\nh$ feedback $K$ is then directly given by
\Cref{thm:lowdatainform}.
Note the difference of \Cref{thm:lowdatainform} to the original data
informativity approach from~\cite{VanETetal20}:
It is not necessarily possible to construct a stabilizing $K$ for all
state-space models in $\Sigma_{\rm{i/s}}$ because it might contain
unstabilizable models due to the non-uniqueness of $\nh$-dimensional models
describing $\nr$-dimensional systems, which prevents the direct
application of \Cref{prp:sysident,prp:constfeed} or
\Cref{cor:constfeed} to $\datatrip$.
Therefore, the additional layer of low-dimensional data and corresponding
state-space models is necessary.

\begin{proof}[Proof of \Cref{thm:lowdatainform}.]
  First, assume that $(A, B) \in \Sigma_{\rm{i/s}}^{\rm{s}}(V)$ is a
  model for which $V$ is not a basis matrix to a left eigenspace of $A$.
  Since~\cref{eqn:CInf:IntDim} holds, we know from \Cref{lmm:minsys} that
  the space spanned by $V$ contains a minimal subspace of dimension $\nmin < \nr$,
  which is a left eigenspace of $A$, such that the dynamics of $(A, B)$ evolve in a
  lower-dimensional subspace.
  The mismatch of this minimal subspace and the one spanned by the columns of $V$ is
  not covered by $V_{\perp}$, i.e., $(A, B)$ has components that do not contribute
  to the dynamics and that are described by a block $A_{33}$
  in~\cref{eqn:extcontform} with arbitrary spectrum.
  Therefore, $(A, B) \in \Sigma_{\rm{i/s}}^{\rm{s}}(V)$ can be chosen unstabilizable
  and, vice versa, due to the corresponding system having dimension
  $\nmin < \nr$, the data matrix $\Xr_{-}$ cannot have full row rank $\nr$ and the
  triplet $\datatripr$ is not informative for stabilization.
  Thus, if $V$ is not a basis matrix of an eigenspace for all $A$
  in $(A, B) \in \Sigma_{\rm{i/s}}^{\rm{s}}(V)$, then there is no $K$ that
  stabilizes all models in $\Sigma_{\rm{i/s}}^{\rm{s}}(V)$ and
  the data triplet $\datatripr$ is not informative for stabilization.

  Now, assume that $V$ is a left eigenbasis matrix for all models
  $(A, B) \in \Sigma_{\rm{i/s}}^{\rm{s}}(V)$.
  With $V_{\perp}$, it holds that
  \begin{align} \label{eqn:lowdatainformTmp1}
    \begin{aligned}
      \begin{bmatrix} V & V_{\perp} \end{bmatrix}^{-1} A
        \begin{bmatrix} V & V_{\perp} \end{bmatrix}
        & = \begin{bmatrix} \Ar & \At \\ 0 & A_{\rm{V}_{\perp}}
        \end{bmatrix}, & 
      \begin{bmatrix} V & V_{\perp} \end{bmatrix}^{-1} B
        & = \begin{bmatrix} \Br \\ 0 \end{bmatrix}, \\
      \begin{bmatrix} V & V_{\perp} \end{bmatrix}^{-1} X_{0}
        & = \begin{bmatrix} \Xr \\ 0 \end{bmatrix},
    \end{aligned}
  \end{align}
  where $X_{0}$ is a basis matrix of the subspace of the initial conditions
  $\Xcal_{0}$.
  The construction of~\cref{eqn:lowdatainformTmp1} follows the use
  of~\cref{eqn:CInf:IntDim} and~\cref{eqn:extcontform}.
  By definition~\cref{eqn:lowdimmodels}, $A_{\rm{V}_{\perp}}$ is stable
  and $(\Ar, \Br) \in \Sigmar_{\rm{i/s}}$.
  The rest of the proof is split into the two implications of the theorem and
  we have without loss of generality that $V$ is a left eigenspace for all $A$ in
  $(A, B) \in \Sigma_{\rm{i/s}}^{\rm{s}}(V)$.
  
  Case~1: Assume the reduced data triplet $\datatripr$ is
  informative for stabilization by feedback.
  Let $\Kr$ be a controller for which $\Sigmar_{\rm{i/s}} \subseteq
  \Sigmar_{\Kr}$ holds.
  Consequently, all $(\Ar, \Br) \in \Sigmar_{\rm{i/s}}$ are stabilizable and,
  with~\cref{eqn:lowdatainformTmp1}, also all $(A, B) \in
  \Sigma_{\rm{i/s}}^{\rm{s}}(V)$ are stabilizable.
  From \Cref{thm:lowdimfeed}, it holds that $K = \Kr V^{\dagger}$ is stabilizing for
  all $(A, B) \in \Sigma_{\rm{i/s}}^{\rm{s}}(V)$ such that
  $\Sigma_{\rm{i/s}}^{\rm{s}}(V) \subseteq \Sigma_{K}$ holds.
  
  Case~2: Assume there exists a $K$ such that $\Sigma_{\rm{i/s}}^{\rm{s}}(V)
  \subseteq \Sigma_{K}$.
  In~\cref{eqn:lowdatainformTmp1} we see that $\Ar + \Br K V$ must then be stable.
  It is left to show that for all $(\Ar, \Br) \in \Sigmar_{\rm{i/s}}$ there exists
  an $(A, B) \in \Sigma_{\rm{i/s}}^{\rm{s}}(V)$ such
  that~\cref{eqn:lowdatainformTmp1} holds, because then $\Sigmar_{\rm{i/s}}
  \subseteq \Sigmar_{KV}$.
  For all $(\Ar, \Br) \in \Sigmar_{\rm{i/s}}$ we have that
  \begin{align*}
    \Xr_{+} & = \Ar \Xr_{-} + \Br U_{-}.
  \end{align*}
  In particular, we can choose any stable $A_{\rm{V}_{\perp}}$ and an
  arbitrary $\At$ such that
  \begin{align*}
    \begin{bmatrix} \Xr_{+} \\ 0 \end{bmatrix} & =
      \begin{bmatrix} \Ar & \At \\ 0 & A_{\rm{V}_{\perp}} \end{bmatrix}
      \begin{bmatrix} \Xr_{-} \\ 0 \end{bmatrix} +
      \begin{bmatrix} \Br_{-} \\ 0 \end{bmatrix} U_{-}.
  \end{align*}
  By multiplication with $\begin{bmatrix} V & V_{\perp} \end{bmatrix}$ from
  the left and using~\cref{eqn:lowdatainformTmp1} it holds
  \begin{align*}
    X_{+} & = \begin{bmatrix} V & V_{\perp} \end{bmatrix}
      \begin{bmatrix} \Ar & \At \\ 0 & A_{\rm{V}_{\perp}} \end{bmatrix}
      \begin{bmatrix} \Xr_{-} \\ 0 \end{bmatrix} +
      \begin{bmatrix} V & V_{\perp} \end{bmatrix}
      \begin{bmatrix} \Br_{-} \\ 0 \end{bmatrix} U_{-}\\
    & = A \begin{bmatrix} V & V_{\perp} \end{bmatrix} 
      \begin{bmatrix} \Xr_{-} \\ 0 \end{bmatrix} +
      B U_{-}\\
    & = A X_{-} + B U_{-},
  \end{align*}
  Therefore, $(A, B) \in \Sigma_{\rm{i/s}}^{\rm{s}}(V)$ and thus the data triplet
  $\datatripr$ is informative for stabilization, which concludes the proof.
\end{proof}

With \Cref{thm:lowdatainform}, the number of data samples necessary for
the construction of guaranteed stabilizing controllers becomes dependent on
the reduced dimension $\nr$ rather than the dimension of the large state
space $\nh$.
This is given in the next corollary.

\begin{corollary}[Reduced number of data samples]%
  \label{cor:numsamp}
  If \Cref{thm:lowdatainform} applies, then the minimum number of
  data samples necessary for the construction of a stabilizing feedback
  controller for all underlying systems reduces to $\nr$, even if
  high-dimensional states of dimension $\nh$ are sampled.
  Also, for unique identification of a state-space model of the underlying
  system, the minimum number of necessary data samples reduces to $\nr + \np$.
\end{corollary}

The dimension $\nr$ plays an essential role in the use of
\Cref{thm:lowdatainform} for the design of stabilizing controllers as
it appears in the rank conditions for data informativity.
In fact, this dimension can be related to the underlying systems that
are stabilized by the feedback, as the following corollary shows.

\begin{corollary}[Minimality of informative dimension]%
  \label{cor:lowdimmin}
  Given the assumptions of \Cref{thm:lowdatainform}.
  If $\datatripr$ is informative for stabilization by feedback, then the
  corresponding state dimension of all
  systems from which the data can be observed is minimal with $\nr = \nmin$. 
\end{corollary}
\begin{proof}
  In the conditions for data informativity in
  \Cref{prp:constfeed,cor:constfeed}, we see that
  the given data matrix $\Xr_{-}$ must have full row rank.
  From~\cref{eqn:extcontform}, we know that uncontrollable parts with zero
  initial conditions do not contribute to the rank of generated data, i.e.,
  the data is full rank if and only if $\nr = \nc + \nx = \nmin$.
  Using \Cref{lmm:minsys} together with
  assumption~\cref{eqn:CInf:IntDim} in \Cref{thm:lowdatainform} gives
  the result.
\end{proof}


\subsection{Controller inference for stabilizing approximately low-dimensional
  systems}%
\label{subsec:apprdata}

In this section, we consider the case where an $\nrr$-dimensional
subspace $\Vcalt \subset \R^{\nh}$ with a basis matrix $\Vt \in
\R^{\nh \times \nrr}$ exists such that the high-dimensional trajectories
that are sampled from the system are well but not exactly represented in~$\Vcalt$:
\begin{align} \label{eqn:approxlowdim}
  \begin{aligned}
    x(t) & \approx \Vt \xt(t), & t \geq 0,
  \end{aligned}
\end{align}
with vectors $\xt(t)$ of dimension $\nrr$.
The vectors $\xt(t)$ are assumed to be states of a low-dimensional state-space
model and, thus, $\xt(t)$ may not be obtained via projection with a left inverse
of the high-dimensional states in general~\cite{Peh20}.

In the following, we argue that in the case of approximately low-dimensional 
systems with condition~\cref{eqn:approxlowdim}, using the
Moore-Penrose inverse $\Vt^{+}$ to lift a low-dimensional controller $\Kt$ can
help to keep the disturbance due to the approximation of $x(t)$ as $\Vt\xt(t)$
low. 
Additionally, we also discuss that it helps to reduce the disturbance in the
feedback if the low-dimensional subspace $\Vcalt$ contains the eigenvectors
corresponding to the unstable eigenvalues of the high-dimensional state-space
model $(A, B)$ from which data $\datatrip$ are sampled.


\subsubsection{Amplification of state approximation errors}

There exists an error vector $x_{\Delta}(t) \in \R^{\nh}$ that closes the gap in approximation~\cref{eqn:approxlowdim}, i.e.,
\begin{align} \label{eqn:distdata}
  \begin{aligned}
    x(t) & = \Vt \xt(t) + x_{\Delta}(t), & \text{for}~t \geq 0.
  \end{aligned}
\end{align}
And vice versa, the low-dimensional state $\xt(t)$ is given by
\begin{align} \label{eqn:disttrunc}
  \begin{aligned}
    \xt(t) & = \Vt^{\dagger}x(t) - \Vt^{\dagger} x_{\Delta}(t), &
      \text{for}~t \geq 0,
  \end{aligned}
\end{align}
for a left inverse $\Vt^{\dagger}$ of $\Vt$.
Equation~\cref{eqn:distdata} states that the reconstruction error of
lifting the low-dimensional state $\xt(t)$ into the high-dimensional space
is given by $x_{\Delta}(t)$, whereas equation~\cref{eqn:disttrunc} states
that the truncation error of approximating $x(t)$ in the reduced space $\Vcalt$
by the low-dimensional state $\xt(t)$ is $-\Vt^{\dagger}x_{\Delta}(t)$.
In particular, if we have a controller $\Kt \in \R^{\np \times \nrr}$
that stabilizes the system corresponding to the low-dimensional states
$\xt(t)$ through feedback 
\begin{align} \label{eqn:LowDimFeedback}
  u(t) & = \Kt \xt(t),
\end{align}
then lifting $\Kt$ gives the feedback
\begin{align*}
  \Kt \Vt^{\dagger} x(t) &
    = \Kt \xt(t) + \Kt \Vt^{\dagger} x_{\Delta}(t)
    = u(t) + \Kt V^{\dagger} x_{\Delta}(t)
\end{align*}
for the high-dimensional states $x(t)$ and $u(t)$ defined
in~\cref{eqn:LowDimFeedback}.
Thus, using the lifted controller $K = \Kt \Vt^{\dagger}$ provides the same
feedback~\cref{eqn:LowDimFeedback} as the reduced state plus the
disturbance $\Kt \Vt^{\dagger}x_{\Delta}(t)$ due to the truncation error.
To understand the performance of the lifted controller $K = \Kt \Vt^{\dagger}$
on the high-dimensional states, we need to understand the effect of
$\Kt \Vt^{\dagger}x_{\Delta}(t)$ on the lifted state.
The first observation is that  $\| \Kt \Vt^{\dagger} x_{\Delta}(t) \|$
should be kept small.
Thus, the Moore-Penrose inverse $\Vt^{+}$ is a good choice because it
minimizes $\| \Vt^{\dagger} x_{\Delta}(t) \|_{2}$ among all possible left
inverses of $V$ and for the unknown errors $x_{\Delta}(t)$.
The second observation follows in the subsequent section.


\subsubsection{Perturbations of the closed-loop spectrum}

A different point to consider is the influence
of the feedback constructed with an approximate subspace on the spectrum of the
state-space model from which data were sampled:
Let $(A, B)$ be a state-space model describing the data triplet
$\datatrip$ with a basis $X_{0}$ of the initial conditions $\Xcal_{0}$.
Under the assumption that no higher-order Jordan blocks are split, there
exists a transformation $S \in \R^{\nh \times \nh}$ such that
\begin{align} \label{eqn:triangdecomp}
  \begin{aligned}
    S^{-1} A S & = \begin{bmatrix} A_{11} & A_{12} \\ 0 & A_{22}
      \end{bmatrix}, &
      S^{-1} B & = \begin{bmatrix} B_{1} \\ B_{2} \end{bmatrix}, &
      S^{-1} X_{0} & = \begin{bmatrix} X_{10} \\ X_{20} \end{bmatrix},
  \end{aligned}
\end{align}
with the matrix blocks $A_{11} \in \R^{\nrr \times \nrr}$,
$A_{12} \in \R^{\nrr \times (\nh - \nrr)}$,
$A_{22} \in \R^{(\nh - \nrr) \times (\nh - \nrr)}$,
$B_{1} \in \R^{\nrr \times \np}$,
$B_{2} \in \R^{(\nh - \nrr) \times \np}$, and the initial conditions
$X_{10} \in \R^{\nrr \times \nq}$ and
$X_{10} \in \R^{(\nh - \nrr) \times \nq}$.
The first block rows and columns in~\cref{eqn:triangdecomp} represent the parts
of the true state-space model, which are approximated
by the model in~\cref{eqn:approxlowdim}.
Due to the assumption that $\xt(t)$ is the state of a linear state-space model,
from~\cref{eqn:triangdecomp}, we can obtain 
\begin{align} \label{eqn:disttriang}
  \begin{aligned}
    \begin{bmatrix} A_{11} & A_{12} \\ 0 & A_{22} \end{bmatrix} &
      = \begin{bmatrix} \At & 0 \\ 0 & 0 \end{bmatrix}
      + \begin{bmatrix} A_{\Delta} & A_{12} \\ 0 & A_{22} \end{bmatrix}, &
      \begin{bmatrix} B_{1} \\ B_{2} \end{bmatrix} &
      = \begin{bmatrix} \Bt \\ 0 \end{bmatrix}
      + \begin{bmatrix} B_{\Delta} \\ B_{2} \end{bmatrix}, \\
    \begin{bmatrix} X_{10} \\ X_{20} \end{bmatrix} &
      = \begin{bmatrix} \Xt \\ 0 \end{bmatrix}
      + \begin{bmatrix} X_{\Delta} \\ X_{20} \end{bmatrix},
  \end{aligned}
\end{align}
where $\At$, $\Bt$ and $\Xt$ define a state-space model for the states in~\cref{eqn:approxlowdim},
and $A_{\Delta}$, $B_{\Delta}$ and $X_{\Delta}$ are appropriate perturbations resulting from the difference to the state-space model of the $\nh$-dimensional states.

Let $\Kt$ be a stabilizing controller constructed for the state-space
model $(\At, \Bt)$.
The full-order closed-loop matrix then reads as
\begin{align} \label{eqn:distclosed}
  \begin{bmatrix} A_{11} & A_{12} \\ 0 & A_{22} \end{bmatrix} +
    \begin{bmatrix} B_{1} \\ B_{2} \end{bmatrix}
    \begin{bmatrix} \Kt & 0 \end{bmatrix}
    & = \begin{bmatrix} \At + \Bt \Kt & A_{12} \\ 0 & A_{22} \end{bmatrix} +
    \begin{bmatrix} A_{\Delta} + B_{\Delta} \Kt & 0 \\ B_{2} \Kt & 0
    \end{bmatrix}.
\end{align}
The sum of the right-hand side in~\cref{eqn:distclosed} separates the main
spectrum and the disturbances.
If the space spanned by the columns of the basis matrix $\Vt$ contains as
subspace the space spanned by the right eigenvectors of $A$ corresponding to the
unstable eigenvalues, then the first term on the right-hand side
of~\cref{eqn:distclosed} is stable by construction of $\Kt$.
This motivates the second observation, namely choosing a space $\Vcalt$ that
contains the right eigenspace of $A$ corresponding to the unstable eigenvalues.

The stability of the closed-loop model~\cref{eqn:distclosed} is disturbed
by $A_{\Delta} + B_{\Delta} \Kt$ and $B_{2} \Kt$.
Let us first consider the disturbance $A_{\Delta} + B_{\Delta} \Kt$:
If $A_{\Delta}$ and $B_{\Delta}$ are small in norm, then this is sufficient for
the disturbance $A_{\Delta} + B_{\Delta} \Kt$ to have little effect on the
spectrum of $\At + \Bt \Kt$.
However, this is not a necessary condition because even if $A_{\Delta}$ and
$B_{\Delta}$ are large in norm, the effect when closing the control loop with
$\Kt$ can be small on the spectrum of $\At + \Bt \Kt$.
Let us now consider the term $B_{2} \Kt$, which introduces disturbances in the
spectrum of~\cref{eqn:distclosed} that are related to un-identified effects of
the controls.
If the approximation is related to sampled data $\datatripr$ and $\nrr$ is
chosen large enough, the norm of $B_{2}$ is typically small compared to
$A_{\Delta} + B_{\Delta} \Kt$ because data are usually collected via
non-zero input signals that excite all controlled components.


\section{Computational procedure for controller inference}%
\label{sec:algorithms}

In this section, we introduce a computational procedure for inferring
stabilizing feedback controllers from data.
The following learning approach is context aware because it learns controllers
directly from data rather than via the detour of system identification.
Thus, the learning takes into account the context of the task of stabilization,
in contrast to the traditional two-step process that first learns a generic
model in ignorance of the actual task of stabilization.
A broader view of the proposed approach through the lens of context-aware
learning is discussed in the conclusions and outlook in
\Cref{sec:conclusions}.


\subsection{Controller inference}

\begin{algorithm}[t]
  \SetAlgoHangIndent{1pt}
  \DontPrintSemicolon
  \caption{Controller inference.}
  \label{alg:infercont}
  
  \KwIn{High-dimensional data triplet $\datatrip$.}
  \KwOut{State-feedback controller $K$.}
  
  Construct a basis matrix $\Vt \in \R^{\nh \times \nrr}$ via the
    singular value decomposition
    \begin{align*}
      \begin{bmatrix} X_{-} & X_{+} \end{bmatrix} & =
        \begin{bmatrix} \Vt & V_{2} \end{bmatrix}
        \begin{bmatrix} \Sigma_{1} & 0 \\ 0 & \Sigma_{2} \end{bmatrix}
        U^{\trans},
    \end{align*}
    where $\Sigma_{1}$ contains the $\nrr$ largest singular values.\;
    
  Compute the reduced data triplet $\datatript$ via
    \begin{align*}
      \begin{aligned}
        \Xt_{-} & = \Vt^{\trans} X_{-} & \text{and} &&
          \Xt_{+} & = \Vt^{\trans} X_{+}.
      \end{aligned}
    \end{align*}
    \vspace{-\baselineskip}\;
    
  Infer a low-dimensional stabilizing feedback $\Kt = U_{-} \Theta
    (\Xt_{-} \Theta)^{-1}$ for $\datatript$ solving either~\cref{eqn:dtlmi}
    or~\cref{eqn:ctlmi} for the unknown $\Theta$.\;
    
  Lift the inferred reduced controller $\Kt$ to the high-dimensional
    space via $K = \Kt \Vt^{\trans}$.\;
\end{algorithm}

An approach building on the theory introduced in \Cref{sec:infercon}
is given in \Cref{alg:infercont}.
In Step~1 of \Cref{alg:infercont}, an orthonormal basis of an
approximation the image of the two concatenated data matrices $X_{-}$ and
$X_{+}$ is computed via the singular value decomposition.
Note that other low-rank matrix approximations such as the pivoted QR
decomposition can be used here as well.
Also note that the dimension $\nrr$ can be chosen based on a suitable energy
measure and the amount of available data samples.

Step~2 computes approximations of the data matrices.
For this step, a left inverse of the basis $\Vt$ is used.
As discussed in previous sections, a suitable choice is the Moore-Penrose
inverse due to its norm minimizing property and stable computability.
Since $\Vt$ is an orthogonal basis matrix in \Cref{alg:infercont}, the
Moore-Penrose inverse is the transpose of $\Vt$.

In Step~3, a reduced stabilizing controller is computed with the reduced data
triplet.
The inference approach in \Cref{prp:constfeed,cor:constfeed}
can be directly applied to the reduced data
triplet $\datatript$ to compute $\Kt$.
This inference step needs at least $\nrr$ data samples.
Notice, however, that Step 3 can be replaced by
controller construction via system identification, if solving the matrix
inequalities is numerically challenging.
If then first a state-space model $(\At, \Bt)$ is learned based on
\Cref{prp:sysident}, then, a stabilizing controller can be designed
for $(\At, \Bt)$ using, for example, pole assignment~\cite{Dat04}, the Bass'
algorithm~\cite{Arm75, ArmR76}, Riccati equations~\cite{LanR95} or partial
stabilization~\cite{BenCQetal00}.
First identifying a reduced state-space model and then constructing a controller
needs at least $\nrr + \np$ data samples based on the proposed approach,
compared to $\nh + \np$; cf. \Cref{prp:sysident}.

Finally, in Step~4, the reduced feedback $\Kt$ is lifted to the full state space
from which the data has been obtained using the same left inverse as for the
truncation of the data, i.e., in our case the transpose of the orthogonal basis
matrix.


\subsection{Data collection via re-projection}

\begin{algorithm}[t]
  \SetAlgoHangIndent{1pt}
  \DontPrintSemicolon
  \caption{Data collection via re-projection.}
  \label{alg:reproj}
  
  \KwIn{Basis matrix $\Vt = \begin{bmatrix} \vt_{1} & \ldots & \vt_{\nrr}
    \end{bmatrix} \in \R^{\nh \times \nrr}$, discretized input signal
    $U_{-} = \begin{bmatrix} u_{1} & \ldots & u_{T} \end{bmatrix} \in
    \R^{\np \times T}$, queryable system
    $F\colon \R^{\nh} \times \R^{\np} \rightarrow \R^{\nh}$.}
  \KwOut{Re-projected data triplet $\datatrip$.}
  
  Initialize $X_{-} = [~]$, $X_{+} = [~]$, $k = 1$.\;
  
  \While{$k \leq T$}{
    \eIf{$k \leq r$}{
      Normalize $\displaystyle x = \frac{\vt_{k}}{\| \vt_{k} \|_{2}}$.\;
    }{
      Compute $\displaystyle \xt = \sum\limits_{j = 1}^{r} \alpha_{j} \vt_{j}$,
        with random coefficients $\alpha_{j}$.\;
        
      Normalize $\displaystyle x = \frac{\xt}{\| \xt \|_{2}}$.\;
    }
    
    Query the system $y = F(x, u_{k})$.\;
    
    Update data matrices $X_{-} = \begin{bmatrix} X_{-} & x \end{bmatrix}$ and
      $X_{+} = \begin{bmatrix} X_{+} & y \end{bmatrix}$.\;
    
    Increment $k \leftarrow k + 1$.\;
  }
\end{algorithm}

At least two numerical issues can arise when using the reduced data triplet
$\datatript$ in \Cref{alg:infercont}.
First, the triplet $\datatript$ may not correspond to a linear time-invariant
system due to the truncation; cf. \Cref{subsec:apprdata}.
Second, the original as well as reduced data can lead to poorly conditioned data
matrices.
A large condition number of the data matrices makes it numerically challenging
to solve the linear matrix inequalities from
\Cref{prp:constfeed,cor:constfeed} for inference.
In \Cref{alg:reproj}, we propose a re-projection scheme to collect data
that heuristically lead to better conditioned data matrix; see~\cite{Peh20} for
details about the re-projection scheme in system identification and
non-intrusive model reduction.
The following re-projection scheme is applicable if the system of interest is
queryable, which means that one can excite the system at feasible inputs and
initial conditions and observe the state trajectory.
\Cref{alg:reproj} can be employed between Steps~1 and~2 of
\Cref{alg:infercont} to generate a re-projected data triplet, which
replaces the original data triplet.

\Cref{alg:reproj} applies a single time step for a state vector from
the approximate reachability subspace $\Vcalt$.
A necessary assumption for this is that the state vector is a feasible initial
condition at which the high-dimensional system can be queried.
For the first $\nrr$ vectors from $\Vcalt$, where $\nrr$ is the dimension of
$\Vcalt$, we can use the columns of the basis matrix $\Vt$.
This means that the first $\nrr$ columns of $X_{-}$ are $\Vt$ such that by
multiplication with its left inverse the first $\nrr$ columns of $\Xt_{-}$
correspond to the $\nrr \times \nrr$ identity matrix.
This is beneficial in terms of conditioning and computational variables for
solvers of linear matrix inequalities needed in
\Cref{prp:constfeed,cor:constfeed}.
The inequalities~\cref{eqn:dtlmi,eqn:ctlmi} cannot be directly
treated by many standard solvers.
They only allow (semi-)definiteness constraints for symmetric or symmetrized
optimization variables (matrices), which does not hold for $\Theta$.
Therefore, we need to introduce the auxiliary variable $Z = X_{-} \Theta$
into~\cref{eqn:dtlmi,eqn:ctlmi}.
This additional layer of variables leads to numerically unsymmetric matrices
$X_{-} \Theta$.
In the case of $X_{-}$ having the identity matrix as a large block, the
multiplication with $\Theta$ can be interpreted as a disturbance of the
optimization variables such that $Z$ and $\Theta$ are in a certain sense
close to each other, which can improve the performance of the solvers.
Additionally, we can expect at least for $X_{-}$ in re-projected form
a lower condition number, which also improves the numerics when solving the
matrix inequalities.

If more than $\nrr$ data samples are needed, then they are generated as linear
combinations of basis vectors of $\Vcalt$ with (normally distributed) random
coefficients so that the random vectors lie in the approximate reachability
subspace $\Vcalt$.
The resulting data triplet $\datatrip$ from \Cref{alg:reproj} is such
that the truncated data triplet $\datatript$ in Step~2 of
\Cref{alg:infercont} is associated to a linear time-invariant system.
Note that this is an assumption of the analysis in
\Cref{subsec:apprdata}.


\section{Numerical examples}%
\label{sec:examples}

In this section, we apply the findings of \Cref{sec:infercon} in terms of
\Cref{alg:infercont,alg:reproj} to design stabilizing
state-feedback controllers for three numerical examples.

We compare the proposed controller inference approach, which
directly learns the controller from data, with stabilizing controllers via
system identification of reduced models.
The controllers are constructed via partial
stabilization~\cite{BenCQetal00} from identified models.
For learning models, we employ
\Cref{prp:sysident} in Step~3 of \Cref{alg:infercont}.
Since we consider in all examples fewer than $\nrr + \np$ data samples, the
identified models and corresponding systems are not unique.
We use the Moore-Penrose inverse in \Cref{prp:sysident} to
compute one specific model for a given data set.
The sampled states are obtained with Gaussian input signals.
If due to instabilities the trajectories diverge to infinity, the data
collection is restarted with a vector from the currently spanned reachability
subspace, i.e., the image of the states observed so far.

\begin{figure}[t]
  \vspace{-\baselineskip}
  \centering
  \tikzexternalenable%
  \tikzsetnextfilename{overview_samples}%
  \begin{tikzpicture}[font = \plotfontsize]
  \begin{axis}[
    width  = .75\textwidth,
    height = .125\textheight,
    scale only axis,
    ybar,
    bar width    = 6pt,
    ymode        = log,
    ymin         = 1,
    ymax         = 8e+03,
    ylabel       = {\#samples},
    ylabel style = {yshift = -.3em},
    xtick             = data,
    symbolic x coords = {synthetic~(DT/CT), heat~flow~(DT), heat~flow~(CT),
      cylinder~wake~(DT)},
    xticklabel style  = {rotate = 0, anchor = north, xshift = 0cm},
    legend columns = 3, 
    legend style   = {
      at     = {(0.5,-0.45)},
      anchor = center,
      /tikz/every even column/.append style = {column sep = 0.25cm}},
    cycle list name = samplelist
  ]
  
    \addplot+[ybar] coordinates{
      (synthetic~(DT/CT), 7)
      (heat~flow~(DT), 4491)
      (heat~flow~(CT), 4491)
      (cylinder~wake~(DT), 6624)
    };
    \addlegendentry{full sys.~id.~($\nh + \np$)}
    
    \addplot+[ybar] coordinates{
      (synthetic~(DT/CT), 5)
      (heat~flow~(DT), 19)
      (heat~flow~(CT), 22)
      (cylinder~wake~(DT), 129)
    };
    \addlegendentry{reduced~sys.~id.~($\nrr + \np$)}
    
    \addplot+[ybar] coordinates{
      (synthetic~(DT/CT), 4)
      (heat~flow~(DT), 18)
      (heat~flow~(CT), 21)
      (cylinder~wake~(DT), 122)
    };
    \addlegendentry{\#samples in experiments ($\dT$)}
    
  \end{axis}
\end{tikzpicture}%
  \tikzexternaldisable%

  \vspace{-.25\baselineskip}
  
  \caption{Number of data samples:
    In all experiments, the number of samples for control
    inference is lower than the number of samples required for learning a
    minimal model of the system.
    Furthermore, the number of samples is orders of magnitude lower than what
    would be required for learning traditional non-minimal models of the same
    dimension as the observed states.  
    This is in agreement with \Cref{thm:lowdatainform}
    and the discussion in \Cref{subsec:apprdata}.}
  \label{fig:overview_samples}
  \vspace{-.75\baselineskip}
\end{figure}
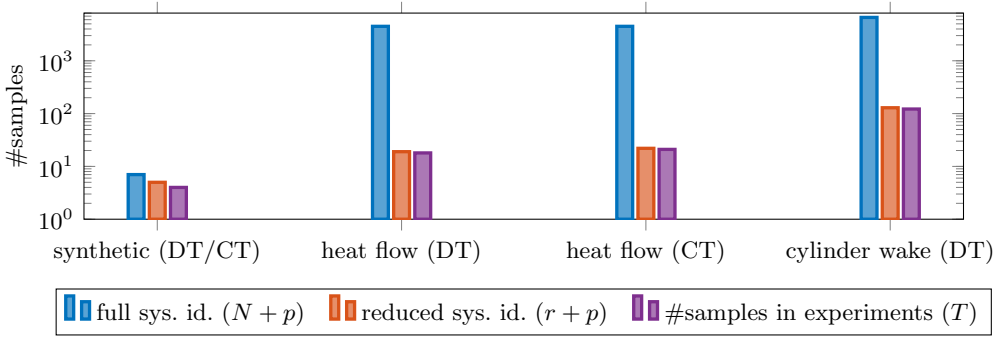

\Cref{fig:overview_samples} provides an overview about the number of data
samples used in the following numerical experiments and how they compare to
traditional two-step approaches that first identify either high- or
low-dimensional models via system identification. 

The experiments have been run on a machine equipped with an Intel(R) Core(TM)
i7-8700 CPU at 3.20GHz and with 16 GB main memory.
The algorithms are implemented in MATLAB 9.9.0.1467703 (R2020b) on CentOS Linux
release 7.9.2009 (Core).
For the solution of linear matrix inequalities, the disciplined convex
programming toolbox CVX version 2.2, build 1148 (62bfcca)~\cite{GraB08, GraB20}
is used together with MOSEK version 9.1.9~\cite{MOS19} as inner optimizer.
For the partial stabilization of identified systems, we use the
implementations of the Bass' algorithm for linear standard systems from the
MORLAB toolbox version~5.0~\cite{BenW21c, BenW19b}.
The code, data and results of the numerical experiments are available
at~\cite{supWer22}.


\subsection{Synthetic example}

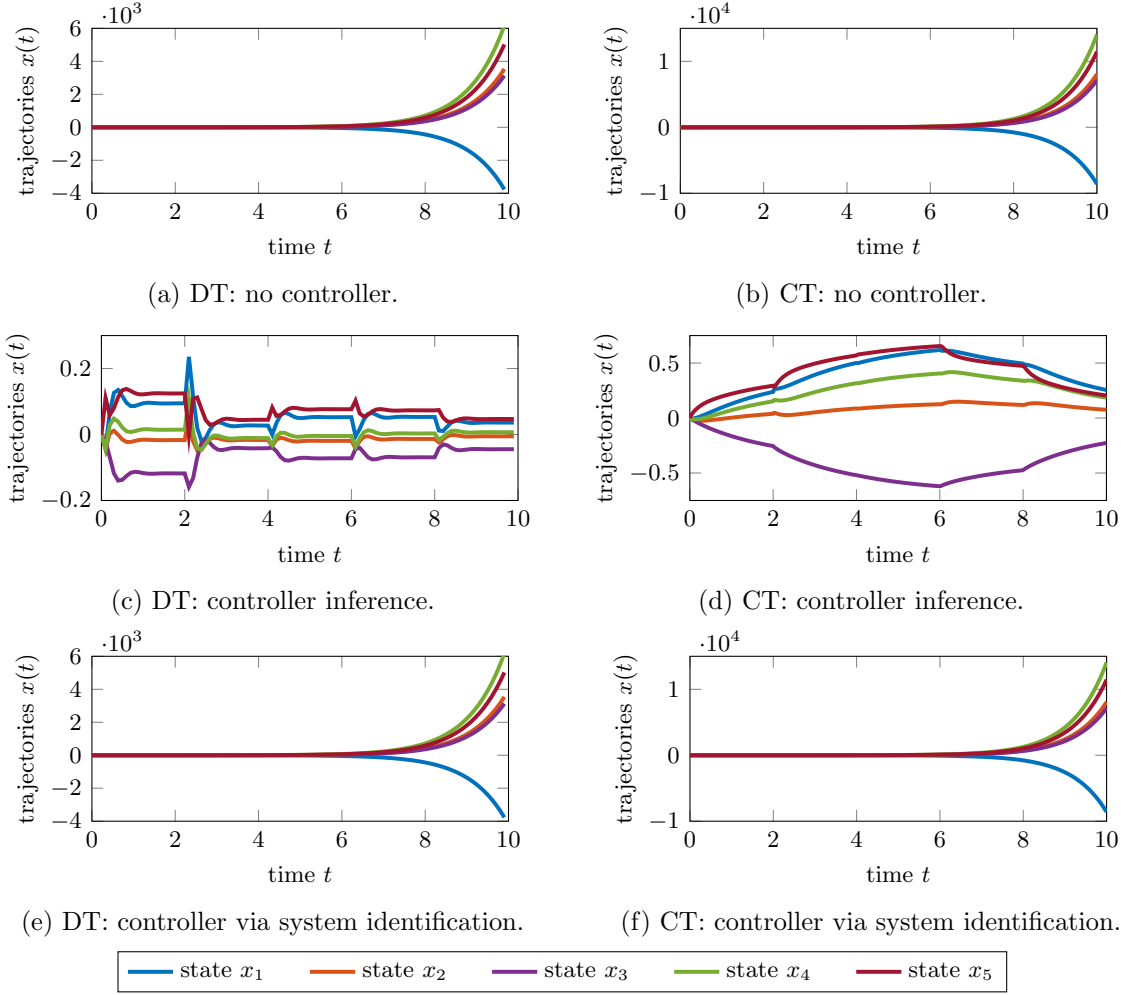
\begin{figure}[t]
  \vspace{-\baselineskip}
  \centering%
  \begin{subfigure}[t]{.49\linewidth}
    \centering
  \tikzexternalenable%
  \tikzsetnextfilename{synthetic_dt_sim_nofb}%
  \begin{tikzpicture}[font = \plotfontsize]
  \pgfplotstableread{graphics/data/synthetic_dt_sim_nofb.dat}\tableSIM
  
  \begin{axis}[%
    name   = states,
    width  = .725\textwidth,
    height = .1\textheight,
    scale only axis,
    xmin = 0,
    xmax = 10,
    ymin = -4000,
    ymax = 6000,
    xminorticks = false,
    yminorticks = false,
    xlabel = {time $t$},
    ylabel = {trajectories $x(t)$},
    ylabel style   = {yshift = -.3em},
    scaled x ticks = false,
    x tick label style = {/pgf/number format/1000 sep={\,}},
    y tick label style = {/pgf/number format/1000 sep={\,}},
    scaled y ticks     = {base 10:-3},
    cycle list name    = stateslist
  ]
  
    \foreach \y in {3, 4, ..., 7}{
      \addplot+ table[x index = 0, y index = \y] {\tableSIM};
    }
  \end{axis}
\end{tikzpicture}%
  \tikzexternaldisable%

    \caption{DT: no controller.}
    \label{fig:synthetic_dt_nofb}
  \end{subfigure}%
  \hfill%
  \begin{subfigure}[t]{.49\linewidth}
    \centering
  \tikzexternalenable%
  \tikzsetnextfilename{synthetic_ct_sim_nofb}%
  \begin{tikzpicture}[font = \plotfontsize]
  \pgfplotstableread{graphics/data/synthetic_ct_sim_nofb.dat}\tableSIM
  
  \begin{axis}[%
    name   = states,
    width  = .725\textwidth,
    height = .1\textheight,
    scale only axis,
    xmin = 0,
    xmax = 10,
    ymin = -1e+04,
    ymax = 1.5e+04,
    xminorticks = false,
    yminorticks = false,
    xlabel = {time $t$},
    ylabel = {trajectories $x(t)$},
    ylabel style   = {yshift = -.3em},
    scaled x ticks = false,
    x tick label style = {/pgf/number format/1000 sep={\,}},
    y tick label style = {/pgf/number format/1000 sep={\,}},
    cycle list name    = stateslist
  ]
  
    \foreach \y in {3, 4, ..., 7}{
      \addplot+ table[x index = 0, y index = \y] {\tableSIM};
    }
  \end{axis}
\end{tikzpicture}%
  \tikzexternaldisable%

    \caption{CT: no controller.}
    \label{fig:synthetic_ct_nofb}
  \end{subfigure}
  \vspace{.25\baselineskip}
  
  \begin{subfigure}[t]{.49\linewidth}
    \centering
  \tikzexternalenable%
  \tikzsetnextfilename{synthetic_dt_sim_inferfb}%
  \begin{tikzpicture}[font = \plotfontsize]
  \pgfplotstableread{graphics/data/synthetic_dt_sim_inferfb.dat}\tableSIM
  
  \begin{axis}[%
    name   = states,
    width  = .725\textwidth,
    height = .1\textheight,
    scale only axis,
    xmin = 0,
    xmax = 10,
    ymin = -0.2,
    ymax = 0.3,
    xminorticks = false,
    yminorticks = false,
    xlabel = {time $t$},
    ylabel = {trajectories $x(t)$},
    ylabel style   = {yshift = -.3em},
    scaled x ticks = false,
    x tick label style = {/pgf/number format/1000 sep={\,}},
    y tick label style = {/pgf/number format/1000 sep={\,}},
    cycle list name    = stateslist
  ]
  
    \foreach \y in {3, 4, ..., 7}{
      \addplot+ table[x index = 0, y index = \y] {\tableSIM};
    }
  \end{axis}
\end{tikzpicture}%
  \tikzexternaldisable%

    \caption{DT: controller inference.}
    \label{fig:synthetic_dt_inferfb}
  \end{subfigure}%
  \hfill%
  \begin{subfigure}[t]{.49\linewidth}
    \centering
  \tikzexternalenable%
  \tikzsetnextfilename{synthetic_ct_sim_inferfb}%
  \begin{tikzpicture}[font = \plotfontsize]
  \pgfplotstableread{graphics/data/synthetic_ct_sim_inferfb.dat}\tableSIM
  
  \begin{axis}[%
    name   = states,
    width  = .725\textwidth,
    height = .1\textheight,
    scale only axis,
    xmin = 0,
    xmax = 10,
    ymin = -0.75,
    ymax = 0.75,
    xminorticks = false,
    yminorticks = false,
    xlabel = {time $t$},
    ylabel = {trajectories $x(t)$},
    ylabel style   = {yshift = -.3em},
    scaled x ticks = false,
    x tick label style = {/pgf/number format/1000 sep={\,}},
    y tick label style = {/pgf/number format/1000 sep={\,}},
    cycle list name    = stateslist
  ]
  
    \foreach \y in {3, 4, ..., 7}{
      \addplot+ table[x index = 0, y index = \y] {\tableSIM};
    }
  \end{axis}
\end{tikzpicture}%
  \tikzexternaldisable%

    \caption{CT: controller inference.}
    \label{fig:synthetic_ct_inferfb}
  \end{subfigure}
  \vspace{.25\baselineskip}
  
  \begin{subfigure}[t]{.49\linewidth}
    \centering
  \tikzexternalenable%
  \tikzsetnextfilename{synthetic_dt_sim_identfb}%
  \begin{tikzpicture}[font = \plotfontsize]
  \pgfplotstableread{graphics/data/synthetic_dt_sim_identfb.dat}\tableSIM
  
  \begin{axis}[%
    name   = states,
    width  = .725\textwidth,
    height = .1\textheight,
    scale only axis,
    xmin = 0,
    xmax = 10,
    ymin = -4000,
    ymax = 6000,
    xminorticks = false,
    yminorticks = false,
    xlabel = {time $t$},
    ylabel = {trajectories $x(t)$},
    ylabel style   = {yshift = -.3em},
    scaled x ticks = false,
    x tick label style = {/pgf/number format/1000 sep={\,}},
    y tick label style = {/pgf/number format/1000 sep={\,}},
    scaled y ticks     = {base 10:-3},
    cycle list name    = stateslist
  ]
  
    \foreach \y in {3, 4, ..., 7}{
      \addplot+ table[x index = 0, y index = \y] {\tableSIM};
    }
  \end{axis}
\end{tikzpicture}%
  \tikzexternaldisable%

    \caption{DT: controller via system identification.}
    \label{fig:synthetic_dt_identfb}
  \end{subfigure}%
  \hfill%
  \begin{subfigure}[t]{.49\linewidth}
    \centering
  \tikzexternalenable%
  \tikzsetnextfilename{synthetic_ct_sim_identfb}%
  \begin{tikzpicture}[font = \plotfontsize]
  \pgfplotstableread{graphics/data/synthetic_ct_sim_identfb.dat}\tableSIM
  
  \begin{axis}[%
    name   = states,
    width  = .725\textwidth,
    height = .1\textheight,
    scale only axis,
    xmin = 0,
    xmax = 10,
    ymin = -1e+04,
    ymax = 1.5e+04,
    xminorticks = false,
    yminorticks = false,
    xlabel = {time $t$},
    ylabel = {trajectories $x(t)$},
    ylabel style   = {yshift = -.3em},
    scaled x ticks = false,
    x tick label style = {/pgf/number format/1000 sep={\,}},
    y tick label style = {/pgf/number format/1000 sep={\,}},
    cycle list name    = stateslist
  ]
  
    \foreach \y in {3, 4, ..., 7}{
      \addplot+ table[x index = 0, y index = \y] {\tableSIM};
    }
  \end{axis}
\end{tikzpicture}%
  \tikzexternaldisable%

    \caption{CT: controller via system identification.}
    \label{fig:synthetic_ct_identfb}
  \end{subfigure}

  \vspace{.25\baselineskip}
  \tikzexternalenable%
  \tikzsetnextfilename{synthetic_legend}%
  \begin{tikzpicture}[font = \plotfontsize]
  \begin{axis}[%
    hide axis,
    width  = .7\textwidth,
    height = .25\textwidth,
    scale only axis,
    xmin = 0,
    xmax = 10,
    ymin = 0.5,
    ymax = 1.5,
    legend columns = 5, 
    legend style = {
      at     = {(0,0)},
      anchor = center,
      /tikz/every even column/.append style = {column sep = 0.5cm}}
  ]
    
    \pgfplotsset{cycle list name = stateslist}
    \pgfplotsinvokeforeach{1, ..., 5}{\addplot coordinates {(0,0)};}
    \addlegendentry{state $x_{1}$}
    \addlegendentry{state $x_{2}$}
    \addlegendentry{state $x_{3}$}
    \addlegendentry{state $x_{4}$}
    \addlegendentry{state $x_{5}$}
  \end{axis}
\end{tikzpicture}%
  \tikzexternaldisable%

  \vspace{-.25\baselineskip}

  \caption{Synthetic example:
    The proposed inference approach leads to stabilizing
    controllers with data sets of only four samples in this example.
    In contrast, the classical two-step control procedure of first identifying a
    model and then constructing a controller leads to unstable dynamics in this
    example because of too few data samples in the data set,
    which is in agreement with \Cref{prp:sysident,thm:lowdatainform}.}
  \label{fig:synthetic}
  \vspace{-.75\baselineskip}
\end{figure}

Consider the system corresponding to the following
state-space model
\begin{align*}
  \begin{aligned}
    A_{0} & = \begin{bmatrix} -4 & 1 & 0 & 0 & 0 \\ 1 & -4 & 1 & 0 & 0 \\
      0 & 1 & 1 & 1 & 0 \\ 0 & 0 & 0 & -4 & 1 \\ 0 & 0 & 0 & 1 & -4 
      \end{bmatrix}, &
      B_{0} & = \begin{bmatrix} 1 & 0 \\ 0 & 1 \\ 0 & 0 \\ 0 & 0 \\ 0 & 0
      \end{bmatrix}\,,
  \end{aligned}
\end{align*}
with state-space dimension $\nh = 5$ and $\np = 2$ inputs.
The dimension of the system is $\nmin = 3$.
The matrix $A_{0}$ is based on a spatial discretization of the two-dimensional
Laplace operator, in which two disturbances are added such that $A_{0}$ is in
upper block triangular form and has one controllable continuous-time
unstable eigenvalue.
The input matrix $B_{0}$ are the first two columns of the identity matrix.
The initial condition is chosen to be homogeneous, $x_{0} = 0$.
To avoid trivially low-dimensional state vectors, the matrices $A_0$ and $B_0$
are transformed by an orthogonal matrix $Q$ with random entries.
The continuous-time version of the considered state-space model is then given
by
\begin{align*}
  \begin{aligned}
    A_{\rm{ct}} & = Q^{\trans} A_{0} Q, &
    B_{\rm{ct}} & = Q^{\trans} B_{0}.
  \end{aligned}
\end{align*}
The transformation does not change the size of the controllable system part, the
eigenvalues of the system matrix or the zero initial condition.
A discrete-time version of the example is obtained using the explicit
Euler scheme with time step size $\tau = 0.1$ on the continuous-time
state-space model such that
\begin{align*}
  \begin{aligned}
        A_{\rm{dt}} & = I_{n} + \tau A_{\rm{ct}}, &
        B_{\rm{dt}} & = \tau B_{\rm{ct}}.
      \end{aligned}
\end{align*}
The discrete-time state-space model has the same zero initial condition,
the same dimension of the controllable system part and also one controllable
unstable eigenvalue.

The trajectories of the discrete- and continuous-time state-space models
are plotted in \Cref{fig:synthetic_dt_nofb,fig:synthetic_ct_nofb}, respectively.
As expected for unstable systems, the trajectories do not converge to a
finite stable behavior but tend to infinity.

We know from \Cref{cor:numsamp,cor:lowdimmin} that
$X_{-}$ from computed data $\datatrip$ needs to have at least rank $3$
for the construction of a stabilizing feedback by informativity, i.e., due
to the homogeneous initial condition, we need overall $\dT = 4$ data samples.
Note that for the identification of a minimal state-space model of the
system, we need at least $\nmin + \np = 5$ data samples.
The numerical rank of the collected data samples in $X_{-}$ and $X_{+}$ is in
agreement with the minimal dimension of the system $\nmin = 3$.
We use \Cref{alg:infercont} to construct stabilizing
feedbacks based on \Cref{thm:lowdatainform} directly from the obtained
data using \Cref{prp:constfeed,cor:constfeed},
i.e., without system identification.
The trajectories corresponding to the stabilized systems are shown in
\Cref{fig:synthetic_dt_inferfb,fig:synthetic_ct_inferfb}.
Due to the closed-loop systems being stable, the trajectories converge to
finite values for the given input signal.
In contrast, because of too little data, the identified reduced state-space
models do not contain unstable eigenvalues such that the feedbacks based on
partial stabilization are zero.
The closed-loop matrices with the identified feedbacks are unstable, as
indicated by the trajectories in
\Cref{fig:synthetic_ct_identfb,fig:synthetic_dt_identfb}.


\subsection{Disturbed heat flow}

\begin{figure}[t]
  \vspace{-\baselineskip}
  \centering
  \begin{subfigure}[b]{.49\linewidth}
    \centering
  \tikzexternalenable%
  \tikzsetnextfilename{hf2d5_dt_sim_nofb}%
  \begin{tikzpicture}[font = \plotfontsize]
  \pgfplotstableread{graphics/data/hf2d5_dt_sim_nofb.dat}\tableSIM
  
  \begin{axis}[%
    name   = states,
    width  = .75\textwidth,
    height = .1\textheight,
    scale only axis,
    xmin = 0,
    xmax = 50,
    ymin = -4e+06,
    ymax = 0,
    xminorticks = false,
    yminorticks = false,
    xlabel = {time $t$},
    ylabel = {outputs $y(t)$},
    ylabel style   = {yshift = -.3em},
    scaled x ticks = false,
    x tick label style = {/pgf/number format/1000 sep={\,}},
    y tick label style = {/pgf/number format/1000 sep={\,}},
    cycle list name    = stateslist
  ]
  
    \foreach \y in {3, 4, ..., 6}{
      \addplot+ table[x index = 0, y index = \y] {\tableSIM};
    }
  \end{axis}
\end{tikzpicture}%
  \tikzexternaldisable%

    \caption{DT: no controller.}
    \label{fig:hf2d5_dt_nofb}
  \end{subfigure}%
  \hfill%
  \begin{subfigure}[b]{.49\linewidth}
    \centering
  \tikzexternalenable%
  \tikzsetnextfilename{hf2d5_ct_sim_nofb}%
  \begin{tikzpicture}[font = \plotfontsize]
  \pgfplotstableread{graphics/data/hf2d5_ct_sim_nofb.dat}\tableSIM
  
  \begin{axis}[%
    name   = states,
    width  = .75\textwidth,
    height = .1\textheight,
    scale only axis,
    xmin = 0,
    xmax = 50,
    ymin = -4e+06,
    ymax = 0,
    xminorticks = false,
    yminorticks = false,
    xlabel = {time $t$},
    ylabel = {outputs $y(t)$},
    ylabel style   = {yshift = -.3em},
    scaled x ticks = false,
    x tick label style = {/pgf/number format/1000 sep={\,}},
    y tick label style = {/pgf/number format/1000 sep={\,}},
    cycle list name    = stateslist
  ]
  
    \foreach \y in {3, 4, ..., 6}{
      \addplot+ table[x index = 0, y index = \y] {\tableSIM};
    }
  \end{axis}
\end{tikzpicture}%
  \tikzexternaldisable%

    \caption{CT: no controller.}
    \label{fig:hf2d5_ct_nofb}
  \end{subfigure}
  \vspace{.25\baselineskip}
  
  \begin{subfigure}[b]{.49\linewidth}
    \centering
  \tikzexternalenable%
  \tikzsetnextfilename{hf2d5_dt_sim_inferfb}%
  \begin{tikzpicture}[font = \plotfontsize]
  \pgfplotstableread{graphics/data/hf2d5_dt_sim_inferfb.dat}\tableSIM
  
  \begin{axis}[%
    name   = states,
    width  = .75\textwidth,
    height = .1\textheight,
    scale only axis,
    xmin = 0,
    xmax = 50,
    ymin = -150,
    ymax = 150,
    xminorticks = false,
    yminorticks = false,
    xlabel = {time $t$},
    ylabel = {outputs $y(t)$},
    ylabel style   = {yshift = -.3em},
    scaled x ticks = false,
    x tick label style = {/pgf/number format/1000 sep={\,}},
    y tick label style = {/pgf/number format/1000 sep={\,}},
    cycle list name    = stateslist
  ]
  
    \foreach \y in {3, 4, ..., 6}{
      \addplot+ table[x index = 0, y index = \y] {\tableSIM};
    }
  \end{axis}
\end{tikzpicture}%
  \tikzexternaldisable%

    \caption{DT: controller inference.}
    \label{fig:hf2d5_dt_inferfb}
  \end{subfigure}%
  \hfill%
  \begin{subfigure}[b]{.49\linewidth}
    \centering
  \tikzexternalenable%
  \tikzsetnextfilename{hf2d5_ct_sim_inferfb}%
  \begin{tikzpicture}[font = \plotfontsize]
  \pgfplotstableread{graphics/data/hf2d5_ct_sim_inferfb.dat}\tableSIM
  
  \begin{axis}[%
    name   = states,
    width  = .75\textwidth,
    height = .1\textheight,
    scale only axis,
    xmin = 0,
    xmax = 50,
    ymin = -200,
    ymax = 250,
    xminorticks = false,
    yminorticks = false,
    xlabel = {time $t$},
    ylabel = {outputs $y(t)$},
    ylabel style   = {yshift = -.3em},
    scaled x ticks = false,
    x tick label style = {/pgf/number format/1000 sep={\,}},
    y tick label style = {/pgf/number format/1000 sep={\,}},
    cycle list name    = stateslist
  ]
  
    \foreach \y in {3, 4, ..., 6}{
      \addplot+ table[x index = 0, y index = \y] {\tableSIM};
    }
  \end{axis}
\end{tikzpicture}%
  \tikzexternaldisable%

    \caption{CT: controller inference.}
    \label{fig:hf2d5_ct_inferfb}
  \end{subfigure}
  \vspace{.25\baselineskip}
  
  \begin{subfigure}[b]{.49\linewidth}
    \centering
  \tikzexternalenable%
  \tikzsetnextfilename{hf2d5_dt_sim_identfb}%
  \begin{tikzpicture}[font = \plotfontsize]
  \pgfplotstableread{graphics/data/hf2d5_dt_sim_identfb.dat}\tableSIM
  
  \begin{axis}[%
    name   = states,
    width  = .75\textwidth,
    height = .1\textheight,
    scale only axis,
    xmin = 0,
    xmax = 50,
    ymin = -4000,
    ymax = 4000,
    xminorticks = false,
    yminorticks = false,
    xlabel = {time $t$},
    ylabel = {outputs $y(t)$},
    ylabel style   = {yshift = -.3em},
    scaled x ticks = false,
    x tick label style = {/pgf/number format/1000 sep={\,}},
    y tick label style = {/pgf/number format/1000 sep={\,}},
    scaled y ticks     = {base 10:-3},
    cycle list name    = stateslist
  ]
  
    \foreach \y in {3, 4, ..., 6}{
      \addplot+ table[x index = 0, y index = \y] {\tableSIM};
    }
  \end{axis}
\end{tikzpicture}%
  \tikzexternaldisable%

    \caption{DT: controller via system identification.}
    \label{fig:hf2d5_dt_identfb}
  \end{subfigure}%
  \hfill%
  \begin{subfigure}[b]{.49\linewidth}
    \centering
  \tikzexternalenable%
  \tikzsetnextfilename{hf2d5_ct_sim_identfb}%
  \begin{tikzpicture}[font = \plotfontsize]
  \pgfplotstableread{graphics/data/hf2d5_ct_sim_identfb.dat}\tableSIM
  
  \begin{axis}[%
    name   = states,
    width  = .75\textwidth,
    height = .1\textheight,
    scale only axis,
    xmin = 0,
    xmax = 50,
    ymin = -2.5e12,
    ymax = 0,
    xminorticks = false,
    yminorticks = false,
    xlabel = {time $t$},
    ylabel = {outputs $y(t)$},
    ylabel style   = {yshift = -.3em},
    scaled x ticks = false,
    x tick label style = {/pgf/number format/1000 sep={\,}},
    y tick label style = {/pgf/number format/1000 sep={\,}},
    cycle list name    = stateslist
  ]
  
    \foreach \y in {3, 4, ..., 6}{
      \addplot+ table[x index = 0, y index = \y] {\tableSIM};
    }
  \end{axis}
\end{tikzpicture}%
  \tikzexternaldisable%

    \caption{CT: controller via system identification.}
    \label{fig:hf2d5_ct_identfb}
  \end{subfigure}

  \vspace{.25\baselineskip}
  \tikzexternalenable%
  \tikzsetnextfilename{hf2d5_legend}%
  \begin{tikzpicture}[font = \plotfontsize]
  \begin{axis}[%
    hide axis,
    width  = .75\textwidth,
    height = .1\textheight,
    scale only axis,
    xmin = 0,
    xmax = 10,
    ymin = 0.5,
    ymax = 1.5,
    legend columns = 5, 
    legend style = {
      at     = {(0,0)},
      anchor = center,
      /tikz/every even column/.append style = {column sep = 0.5cm}}
  ]
    
    \pgfplotsset{cycle list name = stateslist}
    \pgfplotsinvokeforeach{1, 2, 3, 4}{\addplot coordinates {(0,0)};}
    \addlegendentry{output $y_{1}$}
    \addlegendentry{output $y_{2}$}
    \addlegendentry{output $y_{3}$}
    \addlegendentry{output $y_{4}$}
  \end{axis}
\end{tikzpicture}%
  \tikzexternaldisable%

  \vspace{-.25\baselineskip}

  \caption{Heat flow:
    The controllers constructed via the inference approach stabilize the
    system.
    In contrast, applying the traditional two-step approach of first identifying
    a model and then controlling to the same data set only manages to decrease
    the growth of outputs but fails to stabilize the system in the discrete-time
    case.
    In case of the continuous-time system, the controller obtained from the
    identified model even accelerates the growth of the outputs and thus further
    destabilizes the system.}
  \label{fig:hf2d5}
  \vspace{-.75\baselineskip}
\end{figure}
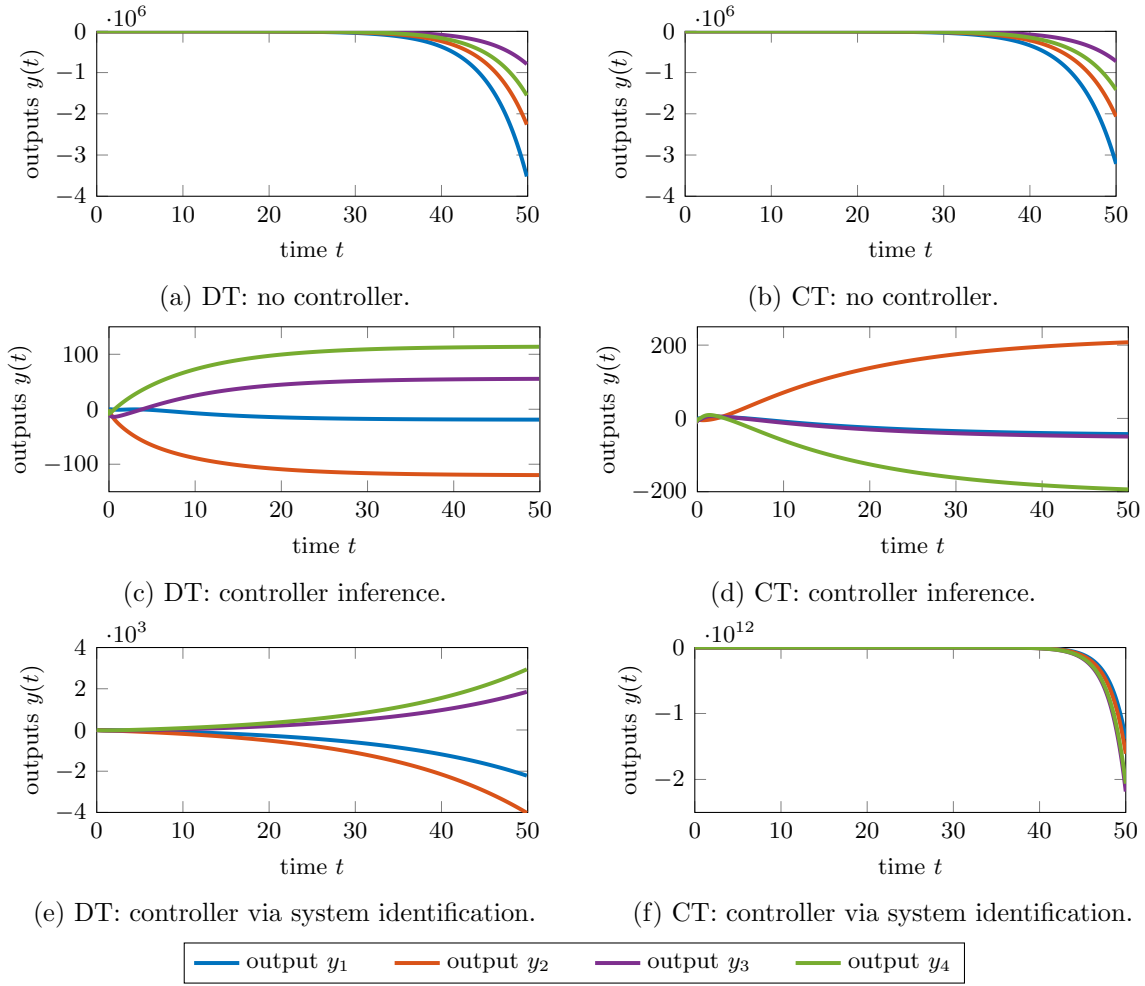

Consider now the system \sf{HF2D5} described in~\cite{Lei04}.
It corresponds to a $2$-dimensional linear heat flow describing the heating
process in a rectangular domain affected by disturbances;
see~\cite[Sec.~3]{Lei04}.
The spatial finite difference discretization yields a high-dimensional
state-space model of dimension $\nh = 4\,489$ with $\np = 2$ inputs
and zero initial conditions.
A discrete-time version of the model is obtained by using the implicit
Euler discretization with time step size $\tau = 0.1$.
The discrete- and continuous-time versions of the model have a single
unstable eigenvalue due to the modeled disturbance in the heating process.
The measured outputs of the resulting time simulations using a unit step
input signal are shown in \Cref{fig:hf2d5_dt_nofb,fig:hf2d5_ct_nofb}.

We employ \Cref{alg:infercont} in the approximate sense as discussed in
\Cref{subsec:apprdata}.
We computed $17$ samples in the discrete-time case to obtain an
approximating subspace $\Vcalt$ of dimension $\nrr = 17$.
In the continuous-time case, we computed $20$ samples, which, due to the
concatenation of the states and their time derivatives, resulted in a subspace
of dimension $\nrr = 21$.
In both cases, the subspaces are constructed to approximate the sampled states
up to machine precision.
We employ the re-projection approach from \Cref{alg:reproj}
to get another data triplet $\datatript$ for the computations, since this
leads to better numerical behavior of the linear matrix inequality solvers.
In both cases, we computed $\dT = \nrr + 1$ data samples via re-projection.
Note that at least $\nrr + \np = \nrr + 2$ samples are necessary for the unique
identification of a reduced state-space model, i.e., the data set contains too
few samples for system identification.

The controller inference leads in both cases to feedbacks that stabilize
the systems.
The corresponding simulations are shown in
\Cref{fig:hf2d5_dt_inferfb,fig:hf2d5_ct_inferfb}.
In contrast, the identified discrete-time reduced model has one unstable
controllable eigenvalue like the original system, which is then stabilized via
partial stabilization.
Applying the lifted controller to the original system shifts
the unstable eigenvalue closer to the unit circle; however, the shift is
insufficient to stabilize the system.
This can be seen in \Cref{fig:hf2d5_dt_identfb}, where the outputs are
still diverging but slower than when no controller is applied.
In the continuous-time case, the identified model has two complex
conjugate unstable eigenvalues, which indicates that a different underlying
system than the true one is approximated.
The constructed feedback stabilizes the learned reduced model but, if applied
to the true system, further destabilizes the system as shown
in \Cref{fig:hf2d5_ct_identfb}.


\subsection{Unstable laminar flow in a cylinder wake}

\begin{figure}
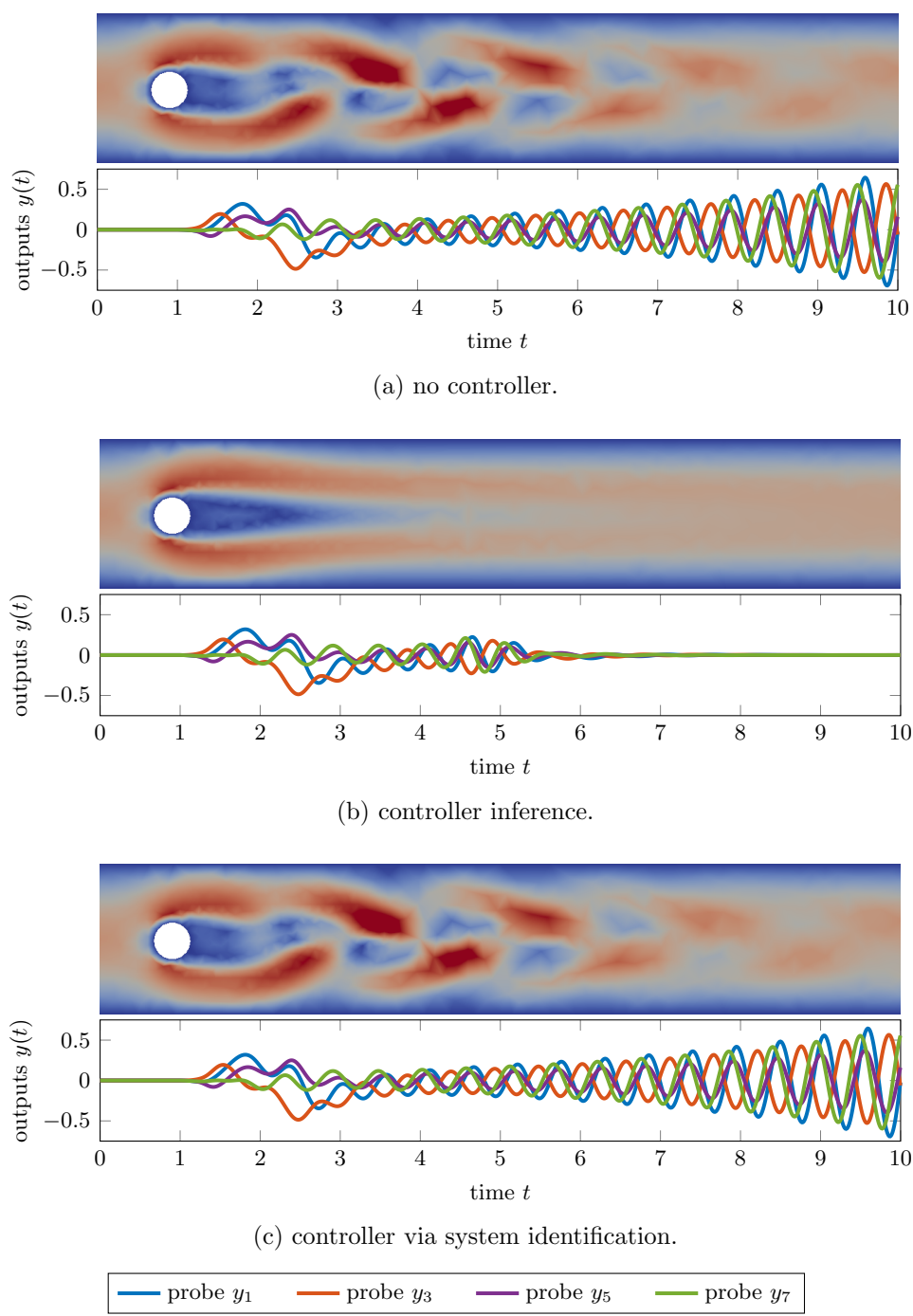

  \centering
  \begin{subfigure}[b]{\linewidth}
    \centering
  \tikzexternalenable%
  \tikzsetnextfilename{cylinderwake_dt_sim_nofb}%
  \input{graphics/cylinderwake_dt_sim_nofb.tikz}%
  \tikzexternaldisable%

    \caption{no controller.}
    \label{fig:cylinderwake_dt_nofb}
  \end{subfigure}
  \vspace{0\baselineskip}

  \begin{subfigure}[b]{\linewidth}
    \centering
  \tikzexternalenable%
  \tikzsetnextfilename{cylinderwake_dt_sim_inferfb}%
  \input{graphics/cylinderwake_dt_sim_inferfb.tikz}%
  \tikzexternaldisable%

    \caption{controller inference.}
    \label{fig:cylinderwake_dt_inferfb}
  \end{subfigure}
  \vspace{0\baselineskip}

  \begin{subfigure}[b]{\linewidth}
    \centering
  \tikzexternalenable%
  \tikzsetnextfilename{cylinderwake_dt_sim_identfb}%
  \input{graphics/cylinderwake_dt_sim_identfb.tikz}%
  \tikzexternaldisable%

    \caption{controller via system identification.}
    \label{fig:cylinderwake_dt_identfb}
  \end{subfigure}
  
  \vspace{.5\baselineskip}
  \tikzexternalenable%
  \tikzsetnextfilename{cylinderwake_legend}%
  \begin{tikzpicture}[font = \plotfontsize]
  \begin{axis}[%
    hide axis,
    width  = .75\textwidth,
    height = .1\textheight,
    scale only axis,
    xmin = 0,
    xmax = 10,
    ymin = 0.5,
    ymax = 1.5,
    legend columns = 5, 
    legend style = {
      at     = {(0,0)},
      anchor = center,
      /tikz/every even column/.append style = {column sep = 0.5cm}}
  ]
    
    \pgfplotsset{cycle list name = stateslist}
    \pgfplotsinvokeforeach{1, 2, 3, 4}{\addplot coordinates {(0,0)};}
    \addlegendentry{probe $y_{1}$}
    \addlegendentry{probe $y_{3}$}
    \addlegendentry{probe $y_{5}$}
    \addlegendentry{probe $y_{7}$}
  \end{axis}
\end{tikzpicture}%
  \tikzexternaldisable%

  \vspace{0\baselineskip}
  
  \caption{Flow behind cylinder:
    In each sub figure, the top plot shows the magnitude of the high-dimensional
    state at final time and the bottom plot shows the averaged velocity in
    horizontal direction at four probes.
    Controller inference is able to stabilize the system,
    whereas traditional data-driven control via system identification fails when
    applied to the same data set.}
  \label{fig:cylinderwake_dt}
\end{figure}

We now consider the dynamics of a laminar flow inside a wake
with a circular obstacle; see \Cref{fig:cylinderwake_dt} for
the geometry.
The flow is described by the Navier-Stokes equations.
The steady state is known to behave unstable for medium and higher Reynolds
numbers.
The goal is to stabilize the system such that deviations from the steady
state are steered back using controls in vertical and horizontal directions
exactly behind the obstacle.
We employ the setup from~\cite{BehBH17} at Reynolds number $90$ and
consider a linearization of the Navier-Stokes equations around the desired
steady state such that the linear system describes the deviation.
This example is a linear system with $\nh = 6\,618$ differential-algebraic
equations (DAEs) and $\np = 6$ inputs describing the controls behind the
obstacle in vertical and horizontal directions.
The system has zero initial conditions.
Due to the DAE form of the problem, we cannot directly obtain the time
derivatives of the state.
Therefore, we consider the example only in discrete-time form using the implicit
Euler discretization with time step size $\tau = 0.0025$.
The resulting system has two unstable eigenvalues.
For visualizations, four sensors are used in the back area of the wake
that measure averaged velocities in horizontal and vertical directions.
The trajectories obtained without control and the magnitudes of the state for
the final time step are shown in \Cref{fig:cylinderwake_dt_nofb}.
As input, a disturbance is emulated in the time interval $[1, 2]$ via a
constant Gaussian step signal.

We take $180$ samples to compute an approximation $\Vcalt$ of the
reachability subspace of dimension $\nrr = 123$.
We use \Cref{alg:reproj} to compute $\nrr + 2 = 125$ re-projected
data samples for the design of stabilizing controllers.
The inferred feedback design stabilizes the system and the trajectories, as
shown by the state of the final time step in
\Cref{fig:cylinderwake_dt_inferfb}.
To improve the presentation of the stabilizing effect of the controller, it
is only applied from time step $5$ onwards.
The results show that the system stops oscillating and is steered back to the
steady state.

In contrast, the state-space model we identified for the reduced re-projected
data set does not have any unstable eigenvalues that could be stabilized.
In fact, we know that the unique identification of a model would
need at least four more data samples ($\nrr + \np = 129$); cf.
\Cref{prp:sysident}.
As result, the constructed controller is zero and does not stabilize the
original system as shown in \Cref{fig:cylinderwake_dt_identfb}.


\section{Conclusions and outlook}%
\label{sec:conclusions}

Learning from data becomes an ever more important component of scientific
computing.
Typically, the focus is on learning models of physical systems.
Once a model is learned, classical scientific computing techniques can be
applied to the learned models for solving upstream tasks such as control,
design, and uncertainty quantification.
However, learning models is only a means to an end in these cases.
The ultimate goal is, e.g., finding an optimal design point and a controller,
rather than learning models.
This raises the question if it is necessary to learn models of complex physics
that completely describe the systems of interest if the goal is solving
potentially simpler  upstream tasks.
A similar question is asked in~\cite{Peh19}, which considers Monte Carlo
estimation as the upstream task.
It proposes to learn models specifically for the use as control variates for
variance reduction.
These models then can have large biases, which is not acceptable for making
predictions about the system response but is sufficient for variance reduction.
Another example is the work~\cite{DeHKNetal21} that studies the learning of
operators corresponding to linear Bayesian inverse problems in contrast to first
learning a (forward) model and then inverting with classical scientific
computing methods.

In this manuscript, we studied the task of stabilizing linear time-invariant
systems.
Building on previous work~\cite{VanETetal20}, our finding is that it is
sufficient to have as many samples as the minimal dimension of the system, which
is fewer than the minimal number of samples required for identifying a minimal
model.
Thus, it is unnecessary to learn models of the underlying systems when the task
is stabilization under the assumptions we made, which results in lower data
requirements.
Given these findings, we believe understanding when learning models of systems
is necessary is an important research direction, which is especially critical in
large-scale science and engineering applications where the state dynamics are
complex and data are scarce.


\section*{Acknowledgments}%
\addcontentsline{toc}{section}{Acknowledgments}

The authors acknowledge support from the Air Force Office of Scientific Research
(AFOSR) award FA9550-21-1-0222 (Dr.~Fariba Fahroo).
The second author additionally acknowledges support from the National Science
Foundation under Grant No.~2012250 and Grant No.~1901091.


\addcontentsline{toc}{section}{References}
\bibliographystyle{plainurl}
\bibliography{bibtex/myref}

\end{document}